\documentclass[english
              ,twoside
              ,11pt
              ]{amsart}
\usepackage[T1]{fontenc}
\usepackage{lmodern}
\usepackage[utf8]{inputenc}
\usepackage{babel}
\usepackage{hyperref}
\usepackage{cas-math}
\usepackage{cas-paper}
\usepackage[arrow,cmtip,matrix]{xy}
\newcommand\draftcomment[1]{\NoDraftCommentShouldBeLeft}

\newcommand{\urlprefix}{}

\newenvironment{acknowledgements}{\subsection{Acknowledgements}}{}
\newcommand{\subclass}[1]{}

\DeclareMathOperator{\Ind}{Ind}
\newcommand{\E}{E}
\newcommand{\B}{B}
\DeclareMathOperator{\diag}{diag}
\DeclareMathOperator{\Iso}{Iso}
\DeclareMathOperator{\dist}{dist}
\def\real{\abs}
\newcommand{\wbar}{\overline w}

\newcommand{\nc}{\mathcal N} % neighbourhood complex

\newcommand{\Spherenk}{\Sphere(W_{n,k})}

\begin{document}

\title{The equivariant topology of stable Kneser graphs}

\author{Carsten Schultz}
\address{Institut für Mathematik, MA 6-2\\
Technische Universität Berlin\\
D-10623 Berlin, \hbox{Germany}}
\email{carsten@codimi.de}
\date{29th March 2010}
%\date{\today}

\begin{abstract}
The stable Kneser graph $SG_{n,k}$, $n\ge1$, $k\ge0$, 
introduced by Schrijver~\cite{schrijver},
is a vertex
critical graph with chromatic number~$k+2$, its vertices are certain
subsets of a set of cardinality~$m=2n+k$.
Björner and de
Longueville~\cite{anders-mark} have shown that its box complex is
homotopy equivalent to a sphere, $\Hom(K_2,SG_{n,k})\homot\Sphere^k$.
The dihedral group $D_{2m}$ acts canonically on $SG_{n,k}$, the group
$C_2$ with $2$ elements acts on $K_2$.  We 
almost %%%% !!!!!
determine the $(C_2\times D_{2m})$-homotopy
type of $\Hom(K_2,SG_{n,k})$ and use this to prove the following
results.

The graphs $SG_{2s,4}$ are homotopy test graphs, i.e.\ for every graph
$H$ and $r\ge0$ such that $\Hom(SG_{2s,4},H)$ is $(r-1)$-connected,
the chromatic number $\chi(H)$ is at least $r+6$.  

If $k\notin\set{0,1,2,4,8}$ and $n\ge N(k)$ then $SG_{n,k}$ is not a
homotopy test graph, i.e.\ there are a graph~$G$ and an $r\ge1$ such
that $\Hom(SG_{n,k}, G)$ is $(r-1)$-connected and $\chi(G)<r+k+2$.
\end{abstract}

\maketitle

\section{Introduction}
\subsection{Background}
The subject of topological obstructions to graph colourings was
started when Lov{\'a}sz determined the chromatic number of Kneser graphs in
\cite{lovasz}.

\begin{defn}
Let $n\ge1$, $k\ge0$.  The Kneser graph $KG_{n,k}$ is a graph with
vertices the $n$-element subsets of a fixed set of cardinality
$2n+k$, say
\[V(KG_{n,k})=\set{S\subset\Z_{2n+k}\colon \card S = n}.\]
Two such sets are neighbours in $KG_{n,k}$ if and only if they
are disjoint,
\[E(KG_{n,k})=\set{(S,T)\in V(KG_{n,k})\colon S\intersect T=\emptyset}.\]
\end{defn}
It is easy to see that $KG_{n,k}$ admits a $(k+2)$-colouring,
$\chi(KG_{n,k})\le k+2$. Lovász assigned to each graph~$G$ a
simplicial complex, its \emph{neighbourhood complex} $\nc(G)$
and proved the following two theorems.
\begin{thm}\label{thm:lovasz-crit}
If $G$ is a graph and $r\ge 0$ such that $\nc(G)$ 
is $(r-1)$-connected, then $\chi(G)\ge r+2$.
\end{thm}

\begin{thm}\label{thm:lovasz-conn}
The complex $\nc(KG_{n,k})$ is $(k-1)$-connected.
\end{thm}

These establish $\chi(KG_{n,k})=k+2$ as conjectured by Kneser.

The proof of \prettyref{thm:lovasz-crit} uses the Borsuk--Ulam theorem.
This led B{\'a}r{\'a}ny to a simpler proof of $\chi(KG_{n,k})\ge k+2$,
which does not use any graph complexes but applies the Borsuk--Ulam
theorem more directly~\cite{barany78}.  This proof uses the existence
of certain generic configurations of vectors in~$\R^{k+1}$.  Using a specific configuration of this kind, Schrijver found an induced subgraph of $KG_{n,k}$, the graph $SG_{n,k}$, with the property that already $\chi(SG_{n,k})=k+2$~\cite{schrijver}.
\begin{defn}
The \emph {stable Kneser graph}
$SG_{n,k}$ is the induced subgraph of $KG_{n,k}$ on the vertex set
\[V(SG_{n,k})=\set{S\in V(KG_{n,k})\colon \text{$\set{i,i+1}\not\subset S$
for all $i\in \Z_{2n+k}$}}.\]
The vertices of~$SG_{n,k}$ are called \emph{stable subsets} of $\Z_{2n+k}$.
\end{defn}
Schrijver also proves that the graph $SG_{n,k}$ is vertex critical in
the sense that it becomes $(k+1)$-colourable if an arbitrary vertex is
removed.

A more systematic treatment of topological obstructions to the
existence of graph colourings was suggested by Lov{\'a}sz and started
by Babson and Kozlov~\cite{babson-kozlov-i}.  
For graphs $G$ and $H$, they define a cell complex
$\Hom(H,G)$.  The vertices of $\Hom(H, G)$ are the graph
homomorphisms from $H$ to~$G$.  They introduce the concept of a test graph.
\begin{defn}\label{def:homotopy-test-graph}
A graph~$T$ is a \emph{homotopy test graph} if for all loopless graphs $G$ and 
$r\ge0$ such the $\Hom(T, G)$ is $(r-1)$-connected the inequality
$\chi(G)\ge \chi(T)+r$ holds.
\end{defn}
Since the complex $\Hom(K_2, G)$ is homotopy equivalent to~$\nc(G)$,
\prettyref{thm:lovasz-crit} states that $K_2$ is a homotopy test
graph.  In~\cite{babson-kozlov-i} this result is extended to all
complete graphs and in \cite{babson-kozlov-ii} to odd cycles.  These
proofs show a graph $T$ to be a test graph by studying the spaces
$\Hom(T, K_n)$ and $C_2$-actions on them induced by a $C_2$-action on~$T$.
Here $C_2$ denotes the
cyclic group of order~$2$.
Indeed, for a graph~$T$ with an action of a group~$\Gamma$ one can
define the property of being a \emph{$\Gamma$-test graph}
(\prettyref{def:gamma-test-graph}), which implies being a homotopy test graphs,
and the homotopy test graphs mentioned above are shown to be $C_2$-test graphs. 

In \cite{hom-loop} easier proofs, in particular for odd cycles, are 
obtained by instead studying $\Hom(K_2, T)$ together
with two involutions, one induced by the non-trivial involution of $K_2$ 
and another by an 
involution on~$T$.  This also yielded the somewhat isolated result
that $KG_{2s,2}$ is a $C_2$-test graph.  All known test graphs at that
point were Kneser graphs or stable Kneser graphs, since $KG_{1,k}=SG_{1,k}$
is a complete graph on $k+2$ vertices and $SG_{n,1}$ is a cycle of
length~$2n+1$

In \cite{anton-cas} it was shown that test graphs~$T$ can be obtained
by constructing graphs~$T$ with prescribed topology of
$\Hom(K_2, T)$.

\subsection{Short overview}
We show how the construction by Bárány and Schrijver can be used to
obtain information on the complex $\Hom(K_2, SG_{n,k})$ and to
establish a close connection between the graphs $SG_{n,k}$ and the
Borsuk graphs of a $k$-sphere.  We modify the construction by Bárány
and Schrijver to make these results equivariant with respect to the
automorphism group of $SG_{n,k}$, $n>1$.  

We use the knowledge of the
equivariant topology of $\Hom(K_2, SG_{n,k})$ to show that the graphs
$SG_{2s,4}$ are homotopy test graphs.  This requires studying the
action of a cyclic group different from $C_2$ on the space $\Hom(K_2,
SG_{2s,4})$.  This distinguishes our result from previous proofs that
certain graphs are test graphs and necessitates the use of different
tools from algebraic topology.  

For $k\notin\set{0,1,2,4,8}$ and large~$n$ we show that $SG_{n,k}$ is
not a test graph.  This uses calculations in the $\Z_2$-cohomology
ring of the automorphism group of~$SG_{n,k}$, a dihedral group.  A
first example of a non-test graph was given by Hoory and
Linial~\cite{linial-test-graph}.  Their example happens to be a graph
for which the bound of \prettyref{thm:lovasz-crit} is not sharp.  On
the other hand, $\chi(SG_{n,k})=k+2$ and $\Hom(K_2, SG_{n,k})$ is
homotopy equivalent to a $k$-sphere and hence in some sense nicest
possible.  It turns out that for these graphs the property of being a
test graph depends on the way in which the automorphism group of the
graph acts on that sphere.  Also, for some values of~$k$ the
graph~$SG_{n,k}$ will fail to satisfy
\prettyref{def:homotopy-test-graph} only for large~$r$, for the
example in~\cite{linial-test-graph} it fails already for~$r=1$.

\subsection{Detailed overview and results}
The starting point of our present investigation is the realization that the
proofs of Kneser's conjecture by Lov{\'a}sz and by B{\'a}r{\'a}ny are more
closely related than by the common appearance of the Borsuk--Ulam
theorem.  We have not seen this connection made explicit in the
literature, even though Ziegler combined ideas from both proofs to
obtain a combinatorial proof of Kneser's
conjecture~\cite{z-inventiones} and even though the vertex criticality
of stable Kneser graphs had suggested that the neighbourhood complex
$\nc(SG_{n,k})$ is homotopy equivalent to a $k$-sphere, which was
proved by Björner and de Longueville~\cite{anders-mark}.

Indeed in the proof of \prettyref{thm:lovasz-crit} the assumption that
$\nc(G)$ be $(r-1)$-connected is used to establish the existence of an
equivariant map from an $r$-sphere to a certain deformation retract of
$\nc(G)$ that comes equipped with a $C_2$-action.  
Spheres are considered $C_2$-spaces via
the antipodal map.  This deformation retract is $C_2$-homotopy
equivalent to $\Hom(K_2, G)$.  On the other hand, B{\'a}r{\'a}ny's
construction can be used to produce a $C_2$-map
$\Sphere^k\to_{C_2}\Hom(K_2, KG_{n,k})$ as we will explain in
\prettyref{rem:barany}.  So even though B{\'a}r{\'a}ny's proof does
not mention graph complexes, consequences for them can be derived from
it.  Using the vector configuration from Schrijver's proof, one even
gets a very explicit map $\Sphere^k\to_{C_2}\Hom(K_2, SG_{n,k})$.  We
describe this map in \prettyref{sec:barany}.  It is not difficult to
see that this map is a homotopy equivalence, we show this in
\prettyref{sec:homotopy}.

Since all ``naturally occurring'' graphs which have so far been
identified as test graphs are stable Kneser graphs (the result for
$KG_{2s,4}$ can be derived from one for $SG_{2s,4}$), it is natural to
ask if more or even all stable Kneser graphs are test graphs.  In
\cite{hom-loop} it turned out that for proving a graph~$T$ to be a
test graph it is useful to not only understand $\Hom(K_2, T)$, but also
its automorphisms induced by automorphisms of~$T$.  The dihedral group
$D_{2m}$, symmetry group of the regular $m$-gon, $m=2n+k$,
acts naturally on
$SG_{n,k}$, so we would like to find a $D_{2m}$-action on $\Sphere^k$
and an equivariant map $\Sphere^k\to_{D_{2m}}\Hom(K_2, SG_{n,k})$.  At
first the map $\Sphere^k\to\Hom(K_2, SG_{n,k})$ constructed with the
help of the vector configuration occuring in Schrijver's proof seems
not to exhibit much symmetry.  But that vector configuration uses
points on the moment curve, and fortunately it turns out that
substituting the trigonometric moment curve for it essentially does
the trick.  This modified construction is carried out in
\prettyref{sec:action}.  Indeed, we define an orthogonal right action
of $D_{2m}$ on $\R^{k+1}$, denote $\R^{k+1}$ equipped with this
action by~$W_{n,k}$, and obtain:
\begin{thm*}[\ref{thm:op-on-box}]
Let $n\ge1$, $k\ge0$, and $m=2n+k$. There is a continuous map
$f\colon\Spherenk\to\real{\Hom(K_2, SG_{n,k})}$ which is equivariant
with respect to the actions of $C_2$ and~$D_{2m}$.
\end{thm*}
The $D_{2m}$-representation $W_{n,k}$ gives rise to a
$(k+1)$-dimensional vector bundle~$\xi_{n,k}$ over the classifying
space $\B D_{2m}$, \prettyref{def:xink}.  This in turn defines
Stiefel-Whitney classes $w_i(\xi_{n,k})\in H^i(D_{2m};\Z_2)$ and also
classes $\wbar_i(\xi_{n,k})\in H^i(D_{2m};\Z_2)$ by
$\left(\sum_{i\ge0}w_i\right)\left(\sum_{i\ge0}\wbar_i\right)=1$.  An
approach similar to that in \cite{hom-loop}, but with different
topological tools, then yields:
\begin{thm*}[\ref{cor:test-graph}]
Let $n,k\ge 1$, $m=2n+k$.  If $\wbar_r(\xi_{n,k})\ne0$ for all $r\ge1$, then
$SG_{n,k}$ is a $D_{2m}$-test graph and hence a homotopy test graph.
\end{thm*}
With the help of the calculation of the classes $w_i(\xi_{n,k})$ in
\prettyref{sec:calc-sw} we obtain:
\begin{thm*}[\ref{cor:n,2}, \ref{cor:2s,4}]
Let $n,k\ge0$.  If $k\in\set{0,1,2}$ or if $k=4$ and $n$~is even,
then $SG_{n,k}$ is a homotopy test graph.
\end{thm*}
In the case $k=4$, $n$ even, we cannot obtain the result by considering
a $C_2$-action on $SG_{n,4}$, but we show $SG_{n,4}$ to be a 
$\Gamma$-test graph, where $\Gamma$ is a cyclic group
whose order is the largest power of~$2$ which divides~$2n+4$.
This is the first case, where a group other than $C_2$ has been used
to prove a graph to be a test graph.
Consequently, this result could not have been obtained by a direct
application of the test graph criteria given in~\cite{hom-loop}.

To prove that some stable Kneser graphs are not test graphs, we need a
dual of~\ref{thm:op-on-box}.  To formulate it we use the notion of the
$\eps$-Borsuk graph of a sphere.  Given $\eps>0$ we define the
infinite graph $B_\eps(\Sphere^k)$ whose vertices are the points
of~$\Sphere^k$, with two vertices adjacent if they are antipodal up to
an error of~$\eps$, \prettyref{def:borsuk}.  Then 
\prettyref{thm:op-on-box}
is equivalent to the existence of equivariant graph homomorphisms
$B_\eps(\Sphere(W_{n,k}))\to_{D_{2m}}SG_{n,k}$ for small enough~$\eps$
as explained in \prettyref{rem:B-to-SG}.  We
also obtain the following:
\begin{thm*}[\ref{prop:eps}]
Let $k\ge 0$ and $\eps>0$.  Then for large enough~$n$ and $m=2n+k$
there is an equivariant graph homomorphism
\[SG_{n,k}\to_{D_{2m}} B_\eps(\Sphere(W_{n,k})).\]
\end{thm*}
This result is also dual to \ref{thm:op-on-box} in another way.  The
configuration of vectors on the trigonometric moment curve used in
both constructions realizes the alternating oriented matroid.  The
constuction of the map in \ref{thm:op-on-box} can be described in
terms of the covectors of the alternating oriented matroid.  The
result \ref{prop:eps} depends on properties of the vectors of the
alternating oriented matroid.

Using \ref{prop:eps} together with Schrijver's result that the stable
Kneser graphs are vertex critical and Braun's result~\cite{braun-sg}
that for $n>1$ the automorphism group of~$SG_{n,k}$ is the dihedral
group $D_{2m}$ we obtain:
\begin{thm*}[\ref{prop:not-test-graph}]
Let $k\ge1$.  Then there is an $N>0$ such that for all $n\ge N$ the
following holds:  If there is an $r>0$ such that 
$\wbar_r(\xi_{n,k})=0$ and  $r=1$ or $r\equiv 0\pmod 2$,
then $SG_{n,k}$ is not a homotopy test graph.
\end{thm*}
Here vertex criticality and knowledge of the
automorphism group come into play because of the following result.
\begin{thm*}[\ref{cor:crit-test}]
Let $T$ be a finite, vertex critical graph.  Then $T$~is a homotopy
test graph if and only if $T$ is an $\Aut(T)$-test graph.
\end{thm*}
This indicates that knowledge of the equivariant topology of a vertex
critical graph can lead to a prove that a graph is \emph{not} a
homotopy test graph.  We state a more directly applicable criterion
in \prettyref{thm:crit-test}.  

After more calculations we obtain from \ref{prop:not-test-graph}:
\begin{thm*}[\ref{cor:even-4-wbar}, \ref{cor:even-2-wbar}, \ref{cor:odd-wbar}]
Let $k\ge0$, $k\notin \set{0,1,2,4,8}$.  Then there is an $N>0$
such that for no $n\ge N$ the graph $SG_{n,k}$ is a 
homotopy test graph.

Also, there is an $N>0$ such that for no odd~$n\ge N$ the graph
$SG_{n,8}$ is a homotopy test graph.
\end{thm*}

\begin{acknowledgements}
I would like to thank Raman Sanyal for a helpful discussion regarding
the use of the trigonometric moment curve in \prettyref{sec:action}
and Rade Živaljević for pointing out to me \prettyref{thm:sw} and its
applicability in situations like those in \prettyref{sec:test-graphs}.
I am also grateful for various helpful comments I received on drafts of
this work.
\end{acknowledgements}

\section{Graphs and complexes}\label{sec:graph-complexes}
We introduce and fix notation for the basic objects that we study.
\subsection{The category of graphs}
A graph $G$ consists of a \emph{vertex set} $V(G)$ and a symmetric
binary relation $E(G)\subset V(G)\times V(G)$.  The relation is called
\emph{adjacency}, adjacent vertices are also called \emph{neighbours}
and elements of $E(G)$ \emph{edges}.  We also write $u\nbo v$ for
$(u,v)\in E(G)$.  A graph is called reflexive if the adjacency
relation is reflexive, i.e.\ if for every $v\in V(G)$ we have
$(v,v)\in E(G)$.  An edge of the form $(v,v)$ is called a loop.

A \emph{graph homomorphism} $f\colon G\to H$ is a function $f\colon
V(G)\to V(H)$ between the vertex sets, which preserves the adjacency
relation, $(f(u),f(v))\in E(H)$ for all $(u,v)\in E(G)$.  In the
following we will denote the category of graphs and graph
homomorphisms by~$\mathcal G$.

The graph~$\ug$ consisting of one vertex and a loop is a final object in
the category of graphs.  Any two graphs $G$ and $H$ have a product
$G\times H$ with 
\begin{align*}
V(G\times H)&=V(G)\times V(H),\\
E(G\times H)&=
\set{((u,u'),(v,v'))\colon 
 \text{$(u,v)\in E(G)$, $(u',v')\in E(H)$}}.
\end{align*}
For every graph $G$, the functor $\bullet\times G$ has a right adjoint
$[G,\bullet]$, i.e. there is a natural equivalence
\[
\mathcal{G}(Z\times G, H)\isom\mathcal{G}(Z, [G,H]).
\]
The graph $[G, H]$ is also written $H^G$ and can be realized by
\begin{align*}
V([G,H])&=V(H)^{V(G)},\\
E([G,H])&=\set{(f,g)\colon 
\text{$(f(u),g(v))\in E(H)$ for all $(u,v)\in E(G)$}}.
\end{align*}
In particular the graph homomorphisms from $G$ to~$H$ correspond to
the looped vertices of~$[G,H]$ in accordance with
\[\mathcal G(G, H)
\isom\mathcal G(\ug\times G, H)
\isom\mathcal G(\ug, [G, H]).
\]
\subsection{Complexes and posets}
For graphs $G$ and $H$ we define a poset
\begin{multline*}
\Hom(G, H)=\\
\set{f\in\left(\power(V(H))\wo\set\emptyset\right)^{V(G)}
\colon
\text{$f(u)\times f(v)\subset E(H)$ f.\ a.\ $(u,v)\in E(G)$}
}
\end{multline*}
with $f\le g$ if and only if $f(u)\subset g(u)$ for all $u\in V(G)$.
$\Hom$ is a functor from $\mathcal G^{\rm opp}\times\mathcal G$ to the
category of posets and order preserving maps.  $\Hom(G, H)$ is the
face poset of a cell complex first described in
\cite{babson-kozlov-i}.  The special case $\Hom(\mathbf1,H)$ is the poset
of cliques of looped vertices of~$H$.  The atoms of $\Hom(G, H)$ corrspond to the graph homomorphisms from $G$ to~$H$.  More is true:
There is a natural homotopy equivalence
\[\real{\Hom(G, H)}\homot\real{\Hom(\ug, [G, H])}\]
induced by a poset map which preserves atoms and with a homotopy
inverse of the same kind.  Also $\real{\Hom(\ug,\bullet)}$
preserves products up to such an equivalence,
see~\cite{anton-graph-cat}.  A possibly better categorical setup is
discussed in \cite[Sec.~7]{anton-cas}.

More formal
properties of $\Hom$ can be derived from the above facts, in particular
the existence of a map
\[\Hom(G, H)\times\Hom(H, Z)\to\Hom(G, Z)\]
which on atoms corresponds to composition of graph homomorphisms and
has all the expected properties.  Of course such a map is also easy to
write down explicitly, it was first used in~\cite{hom-loop}.

A graph complex that was mentioned in the introduction is the
\emph{neighbourhood complex $\nc (G)$} of a graph~$G$.  It is the
abstract simplicial complex consisting of sets of vertices of~$G$ with
a common neighbour.  It follows that the vertices of~$\nc (G)$ are the
non-isolated vertices of~$G$.  We will use in \prettyref{sec:homotopy}
that $\real{\nc(G)}\homot\real{\Hom(K_2, G)}$.

Another construction that we will use takes a poset $P$ (usually the
face poset of a cell complex) and assigns to it a reflexive graph
$P^1$.  The vertices of~$P^1$ are the atoms of~$P$, and to atoms are
adjacent in $P^1$, if and only if they have a common upper bound in
$P$.  This construction played in important role in~\cite{anton-cas}
and we will use several results from there.

\section{The B\'ar\'any-Shrijver construction}\label{sec:barany}

In this section we define the alternating oriented matroid
$C^{m,k+1}$, and for $m=2n+k$ a poset map $\mathcal
L(C^{m,k+1})\to\Hom(K_2,SG_{n,k})$ from the poset of its covectors to
the face poset of the box complex of a stable Kneser graph.  For a
configuration of $m$ vectors in $\R^{k+1}$ which realizes $C^{m,k+1}$,
this defines a map $\Sphere^k\to\real{\Hom(K_2,SG_{n,k})}$.  In 
\prettyref{sec:action} 
we will use specific vector configurations to obtain the
equivariant maps of \prettyref{thm:op-on-box},
which was mentioned in the introduction.

We define the alternating oriented matroid $C^{m,k+1}$, $m>k\ge0$, 
to be the
oriented matroid associated to the vector
configuration $v_0,\dots,v_{m-1}\in\R^{k+1}$ with
$v_j=(1,t_0,\dots,t_i^k)$ for some real numbers
$t_0<t_1<\dots<t_{m-1}$, see \cite[9.4]{oriented-matroids}.  
The set of non-zero covectors is 
\begin{multline*}
\mathcal
L(C^{m,k+1})=\set{(\sign\inprod x{v_0},\dots,\sign\inprod
  x{v_{m-1}})\colon x\in\R^{k+1}\wo\set0}
\\
\subset\set{-1,0,+1}^m.
\end{multline*}
We regard
$\mathcal L(C^{m,k+1})$ as a poset with the partial order induced by
the partial order on $\set{-1,0,+1}$ given by $s\le s'\iff
s=0\lor s=s'$.  The elements of $\mathcal L(C^{m,k+1})$
are exactly the sign vectors with at most $k$ sign changes.  By this
we mean those sign vectors obtained as $(\sign p(0),\dots,\sign
p(m-1))$ with $p$ a polynomial of degree~$k$.  The minimal elements,
called cocircuits, are those with exactly $k$ zeros (and hence a sign
change at every zero, to be interpreted as above).  We denote the set
of cocircuits by $\mathcal C^\ast(C^{m,k+1})$.

The set of covectors of the oriented matroid determines the set of
non-zero vectors 
\[\mathcal V(C^{m,k+1})
=\set{(\sign\lambda_0,\dots,\sign\lambda_{m-1})
\colon\text{$\lambda\in\R^m\wo\set0$, $\sum_i\lambda_i v_i=0$}
}
\]
of the matroid, which are not to be confused with the elements of the
vector configuration (one usually works with signed points in affine
$k$-space instead) that we started with.  It is the vectors that give
the alternating oriented matroid its name, it is not difficult to
check that $s\in\mathcal V(C^{m,k+1})$ if and only if there are
indices $i_0<\dots<i_{k+1}$ such that $s_{i_j}s_{i_{j+1}}=-1$ for
all~$j$; 
compare the proof of \cite[Prop.~9.4.1]{oriented-matroids}.

\begin{propdef}\label{def:C-to-Hom}
Let $n\ge1$, $k\ge0$ and $m=2n+k$.
For $s=(s_0,\dots,s_{m-1})\in\mathcal L(C^{m,k+1})$ let
sets $S_0(s),S_1(s)\subset\set{0,\dots,m-1}$ be defined by
\[S_l(s)=\set{j\colon (-1)^j s_j=(-1)^{l}}.\]
Then
\begin{align*}
%f\colon
\mathcal L(C^{m,k+1})&\to\Hom(K_2,SG_{n,k})\\
s&\mapsto\left(l\mapsto \set{T\in V(SG_{n,k})\colon T\subset S_l(s)}\right),
\end{align*}
with $V(K_2)=\set{0,1}$,
is a well-defined order preserving map.
\end{propdef}

\begin{rem}\label{rem:SG1k}
It is easy to check that for $n=1$ this map is an isomorphism
$\mathcal L(C^{k+2,k+1})\to\Hom(K_2,SG_{1,k})\isom\Hom(K_2, K_{k+2})$.
\end{rem}

\begin{proof}
Denote the map by $g$.  
For $s\in\mathcal L(C^{m,k+1})$ 
obviously $S_0(s)\intersect S_1(s)=\emptyset$,
so $(T_0,T_1)\in E(SG_{n,k})$ for all 
$T_l\in g(s)_l$.  We only have to check that $g(s)_l\ne\emptyset$,
and since $g$~is order preserving by construction, we can assume
that $s$~is a cocircuit.  But then $s$ contains exactly 
$k$~zeros.  Let $i_0<i_1<\dots<i_{2n-1}$ be the indices at which
$s$~is non-zero.  That $s$ has sign changes at exactly the zeros 
means that $s_{i_{j+1}}=(-1)^{i_{j+1}-i_j-1}s_{i_j}$, i.e.\
$(-1)^{i_{j+1}}s_{i_{j+1}}=-(-1)^{i_j}s_{i_j}$.  Therefore 
$S_0(s)$ and $S_1(s)$ are interleaved $n$-sets and
$S_0(s), S_1(s)\in V(SG_{n,k})$.
\end{proof}

\begin{propdef}\label{def:S-to-Hom}
Let $n\ge1$, $k\ge0$ and $m=2n+k$
and $(v_j)$ the vector configuration above or any
other vector configuration realizing $C^{m,k+1}$.

\sloppy
The covector poset $\mathcal L(C^{m,k+1})$ is the face poset of a cell
decomposition of $\Sphere^k$, where the open cell corresponding to
$s\in\mathcal L(C^{m,k+1})$ is $\set{x\in\Sphere^k\colon
\sign\inprod x{v_j}=s_j}$.  Therefore the poset map of
\prettyref{def:C-to-Hom} induces a continuous map
\[%\real
f\colon\Sphere^k\to\real{\Hom(K_2,SG_{n,k})}.\]
If we write the group with $2$~elements as $C_2=\set{e,\tau}$ and
have $C_2$ operate via the antipodal map on $\Sphere^k$ and via
the isomorphism $\Aut(K_2)\isom C_2$ on $\Hom(K_2,SG_{n,k})$, then
the map~$f$ satisfies $f(\tau\cdot x)=\tau\cdot f(x)$ 
for all~$x\in\Sphere^k$.
\qed
\end{propdef}

\begin{rem}
\label{rem:barany}
B\'ar\'any's proof \cite{barany78}
of the Kneser conjecture $\chi(KG_{n,k})\ge k+2$
constructs a covering of $\Sphere^k$ by $s$~open sets, none of them
containing an antipodal pair of points, from a colouring $c\colon
KG_{n,k}\to K_s$ and then invokes the Borsuk--Ulam theorem to conclude
that $s\ge k+2$.  The covering is defined as
\[\left(
\Union_{S\in c^{-1}[\set j]}U_S
\right)_{j\in V(K_s)}
\text{\quad with\quad}
U_S=\set{x\in\Sphere_k\colon
   \text{$\inprod x{w_i}>0$ f.a.~$i\in S$}},
\]
where $(w_i)$ is a system of $m=2n+k$ vectors in~$\R^{k+1}$.  By
construction $x\in U_S$ and $-x\in U_T$ imply $S\intersect
T=\emptyset$, so no element of the covering does contain an antipodal
pair of points.  On the other hand, the vectors $w_i$ have to be
chosen appropriately so that the sets $U_S$, $S\in V(KG_{n,k})$ cover
the $k$-sphere.  B\'ar\'any uses a theorem by Gale to show that this
is possible.  Now such a vector configuration defines a hyperplane
arrangement in $\R^{k+1}$ and a cell complex which subdivides
$\Sphere^k$.  Let $X$ be this cell complex and $FX$ its face poset.
Then similar to \prettyref{def:C-to-Hom} we can define for a face $s$ and 
a point $x$ in its interior
$S_l(s)\deq\set{i\colon\sign\inprod x{w_i}=(-1)^l}$ and
\begin{align*}
FX &\to \Hom(K_2, KG_{n,k}),\\
s &\mapsto\left(l\mapsto\set{T\in V(KG_{n,k})\colon 
T\subset S_l(s)}\right).
\end{align*}
The two properties of the system $(U_S)$ above that make B\'ar\'any's
proof work also ensure that this is a well-defined poset map. It
induces an equivariant map $\Sphere^k\to_{C_2}\Hom(K_2, KG_{n,k})$.
This is the connection between B\'ar\'any's proof and Lov{\'a}sz' proof
that we mentioned in the introduction.

Schrijver \cite{schrijver} refined B\'ar\'any's construction by
choosing $w_i=(-1)^iv_i$ with $v_i$ on the moment curve as in the
beginning of this section.  With this choice already $(U_S)_{S\in
V(SG_{n,k})}$ is a covering of $\Sphere^k$.  Accordingly, we obtain
the map $\Sphere^k\to_{C_2}\Hom(K_2, SG_{n,k})$
of \prettyref{def:S-to-Hom}.

Ziegler \cite{z-inventiones} used the combinatorics of the cell
complex obtained by the hypersphere arrangement corresponding to the
vectors~$v_i$ to produce a combinatorial proof of $\chi(SG_{n,k})\ge
k+2$.  He does not mention graph complexes in this proof, but for a
given colouring $c\colon SG_{n,k}\to K_s$ defines a map which is
essentially the composition
\[\mathcal L(C^{m,k+1})\to \Hom(K_2, SG_{n,k})
 \xto{\Hom(K_2, c)} \Hom(K_2, K_s)\isom\mathcal L(C^{s,s-1}),\]
where the first map is our map of
\prettyref{def:C-to-Hom} and the isomorphism that of 
\prettyref{rem:SG1k}.
\end{rem}

\section{The action of the dihedral group}\label{sec:action}
\begin{defn}\label{def:groups-and-actions}
For $m\ge2$ let
\begin{equation*}
D_{2m}=\left\langle \sigma,\rho \,|\, \rho^2=\sigma^m=(\sigma\rho)^2=1
\right\rangle
\end{equation*}
denote the dihedral group with $2m$ elements.

For $m=2n+k$ we define a right $D_{2m}$ action on $KG_{n,k}$ by
\begin{align*}
S\cdot\sigma&=\set{j+1\colon j\in S},\\
S\cdot\rho&=\set{-j\colon j\in S},
\end{align*}
where all arithmetic is modulo~$m$.  The subgraph $SG_{n,k}$ is
invariant under this action.

We also set
\begin{equation*}
C_2=\left\langle \tau\,|\,\tau^2=1\right\rangle
\end{equation*}
and have this group act nontrivially on~$K_2$ from the right.

This defines a left $C_2$-action and a right $D_{2m}$-action on
$\Hom(K_2, SG_{n,k})$, and these commute.
\end{defn}

Our goal is to choose vectors $v_i$ in \prettyref{def:S-to-Hom} in
such a way that that map becomes $D_{2m}$-equivariant with respect to
an easy to define $D_{2m}$-action on~$\Sphere^k$.  It turns out that
the following definition achieves this.

\begin{defn}\label{def:Wnk}
Let $n\ge 1$, $k\ge0$, $m=2n+k$,
We define orthogonal actions on $\R^{k+1}$ as follows.  
First we set
\begin{align}\nonumber
\tau\cdot x&=-x.
\intertext{For $k=2r$ we set}
\label{eq:def-sigma-even}
x\cdot\sigma&=
-\diag(
1,
R_{2\pi/m},
R_{4\pi/m},
\dots,
R_{k\pi/m})\cdot x
\intertext{with}\nonumber
R_\phi&=\begin{pmatrix}
\cos\phi&-\sin\phi\\
\sin\phi&\cos\phi
\end{pmatrix}
\intertext{and}
\label{eq:def-rho-even}
x\cdot\rho&=
\diag(
1,
1,
-1,
\dots,
1,
-1)\cdot
x.
\intertext{For $k=2r+1$ we set}
\label{eq:def-sigma-odd}
x\cdot\sigma&=
-\diag(
R_{\pi/m},
R_{3\pi/m},
\dots,
R_{k\pi/m})\cdot
x
\intertext{and}
\label{eq:def-rho-odd}
x\cdot\rho&=
\diag(
1,
-1,
\dots,
1,
-1)\cdot
x.
\end{align}

%\begin{nota}
We denote $\R^{k+1}$ equipped with the
orthogonal right action of $D_{2m}$ defined above by $W_{n,k}$.  The
unit sphere in $W_{n,k}$ is denoted by $\Sphere(W_{n,k})$.
%\end{nota}
\end{defn}

We will spend the rest of this section on the construction of a map
satisfying the following theorem.

\begin{thm}\label{thm:op-on-box}
Let $n\ge1$, $k\ge0$, and $m=2n+k$. There is a continuous map
$f\colon\Spherenk\to\real{\Hom(K_2, SG_{n,k})}$ which is equivariant
with respect to the actions of $C_2$ and~$D_{2m}$ defined above.
\end{thm}

To prove this theorem we will define vectors $v_j\in\R^{k+1}$, $j\in\Z$, such
that the system $(v_j)_{0\le j<m}$ realizes $C^{m,k+1}$ and that
additionally
\begin{equation}\label{eq:v-props}
\begin{aligned}
v_j\cdot\sigma=&-v_{j+1},
&
v_j\cdot\rho&=v_{-j},
&
v_{j+m}&=(-1)^m v_j
\end{aligned}
\end{equation}
for all~$j\in\Z$.  Then the expression $(-1)^jv_j$ is
well-defined for $j\in\Z_m$ and it follows that
\begin{align*}
(-1)^{j+1}\inprod{x\sigma}{v_{j+1}}
&=(-1)^j\inprod{x\sigma}{v_j\sigma}=(-1)^j\inprod x{v_j},
\\
(-1)^{-j}\inprod{x\rho}{v_{-j}}
&=(-1)^j\inprod{x\rho}{v_j\rho}=(-1)^j\inprod x{v_j},
\end{align*}
which shows that the map~$f$ arising from this 
vector configuration is $D_{2m}$-equivariant.

\subsection{The case of an even~$k$}
For $k=2r$ we identify $\R^{k+1}=\R\times\C^r$ and set
\begin{equation}\label{eq:v-even}
\begin{split}
v_j&=(1,\exp(2\pi ij/m),\exp(4\pi ij/m),\dots,\exp(2r\pi ij/m))
\\&=(1,\xi^j,\xi^{2j},\dots,\xi^{rj})
\end{split}
\end{equation}
with $\xi=\exp(2\pi i/m)$.  These vectors satisfy
\prettyref{eq:v-props}.

\begin{lem}
The system $(v_j)_{0\le j<m}$ of \prettyref{eq:v-even} realizes
$C^{m,k+1}$.
\end{lem}

\begin{proof}
It suffices to show that every cocircuit has at most $k$ zeros and
changes the sign at every zero.  
For $x=(x_0,x_1,\dots,x_r)\in\R\times\C^r$ we have
\begin{align*}
\inprod{x}{v_j}_\R
&=\frac12\left(x_r\xi^{-rj}+\dots+x_1\xi^{-j}
+2x_0+\overline{x_1}\xi^j+\dots+\overline{x_r}\xi^{rj}\right).
\end{align*}
Now if we consider the function
$h(t)\deq 2\inprod x{\exp(2\pi it/m)}_\R$ then
$2\inprod x{v_j}=h(j)$ and
\begin{align*}
h(t)&=\exp(-2\pi irt/m)p(\exp(2\pi it/m))
\intertext{with}
p(z)&=
x_r+x_{r-1}z+\dots+x_1z^{r-1}+2x_0z^r
+\overline{x_1}z^{r+1}+\dots+\overline{x_r}z^{2r}.
\end{align*}
Now $\inprod x{v_j}=0$ if and only if $p(\xi^j)=0$,
but $p$~is a polynomial of degree $k$ and $\xi^j\ne\xi^{j'}$
for $0\le j<j'<m$.
Also
\[h'(t)=2\pi i/m\exp(-2\pi irt/m)
\left(-rp(\exp(2\pi it/m))+p'(\exp(2\pi it/m))\right)\]
and therefore $h(t)$ changes the sign 
whenever $\exp(2\pi it/m)$ is a simple root of~$p$.
\end{proof}

\subsection{The case of an odd~$k$}
For $k=2r+1$ we identify $\R^{k+1}=\C^{r+1}$ and set
\begin{equation}\label{eq:v-odd}
\begin{split}
v_j&=(\exp(\pi j/m),\exp(3\pi j/m),\dots,\exp((2r+1)\pi j/m))
\\&=(\xi^j,\xi^{3j},\dots,\xi^{(2r+1)j})
\end{split}
\end{equation}
with $\xi=\exp(\pi i/m)$.
These vectors, too, satisfy \prettyref{eq:v-props}.

\begin{lem}
The system $(v_j)_{0\le j<m}$ of \prettyref{eq:v-odd} realizes
$C^{m,k+1}$.
\end{lem}

\begin{proof}
Again, it suffices to show that every cocircuit has at most $k$ zeros and
changes the sign at every zero.  
For $x=(x_0,x_1,\dots,x_r)\in\C^{r+1}$ we have
\begin{align*}
\inprod{x}{v_j}_\R
&=\frac12\left(x_r\xi^{-(2r+1)j}+\dots+x_0\xi^{-j}
+\overline{x_0}\xi^j+\dots+\overline{x_r}\xi^{(2r+1)j}\right).
\end{align*}
Now if we consider the function
$h(t)\deq 2\inprod x{\exp(\pi it/m)}_\R$ then
$2\inprod x{v_j}=h(j)$ and
\begin{align*}
h(t)&=\exp(-(2r+1)\pi it/m)p(\exp(2\pi it/m))
\intertext{with}
p(z)&=
x_r+x_{r-1}z+\dots+x_0z^{r}
+\overline{x_0}z^{r+1}+\dots+\overline{x_r}z^{2r+1}.
\end{align*}
Now $\inprod x{v_j}=0$ if and only if $p(\xi^{2j})=0$,
but $p$~is a polynomial of degree $k$ and
$\xi^{2j}\ne\xi^{2j'}$ for $0\le j<j'<m$.
Also
\[h'(t)=\pi i/m\cdot\exp(-k\pi it/m)
\left(-k p(\exp(2\pi it/m))+2p'(\exp(2\pi it/m))\right)\]
and therefore $h(t)$ changes the sign 
whenever $\exp(2\pi it/m)$ is a simple root of~$p$.
\end{proof}
This concludes the proof of \prettyref{thm:op-on-box}.
\qed

\section{Homotopy}\label{sec:homotopy}
The results of this section are not needed in the rest of this paper.  

The construction of \prettyref{sec:barany} leads to a more conceptual
proof of the homotopy equivalence $\nc(SG_{n,k})\homot\Sphere^k$
of~\cite{anders-mark} and,
using the result of \prettyref{sec:action}, 
also of the following $D_{2m}$-equivariant
version.
\begin{thm}\label{thm:homotopy}
Let $n\ge1$, $k\ge0$ and $m=2n+k$.  Then
the map $f\colon\Spherenk\to\real{\Hom(K_2, SG_{n,k})}$ constructed in the
proof of \prettyref{thm:op-on-box} is a $D_{2m}$-homotopy equivalence.
\end{thm}
Of course, this map is also a $C_2$-homotopy equivalence, since it is
$C_2$-equivariant and the $C_2$-actions on both spaces are free.
We do not see a an easy way to extend this to a proof of $(C_2\times
D_{2m})$-homotopy equivalence, however.

To not obfuscate the basic idea with technical details, we start with
a sketch of a proof of the weaker statement 
$\real{\nc(SG_{n,k})}\homot\Sphere^k$.
We take up the 
notation from \prettyref{rem:barany}
and observe the following.
\begin{prop}\label{prop:nerve}
Let $(v_i)$ be a system of vectors realizing $C^{m,k+1}$ and 
\begin{equation*}
U_S\deq\set{x\in\Sphere^k\colon\inprod{x}{(-1)^j v_j}>0\text{\ f.a. $j\in S$}}
\end{equation*}
for $S\subset\set{0,\dots,m-1}$.
Then 
the nerve of the system $(U_S)_{S\in V(SG_{n,k})}$ of open sets is the simplicial 
complex $\nc(SG_{n,k})$.
\end{prop}

\begin{proof}
Let $\emptyset\ne\mathcal S\subset V(SG_{n,k})$.  We have to show that
$\Intersection_{A\in\mathcal S} U_S\ne\emptyset$ if and only if there
is a $T\in V(SG_{n,k})$ which is a common neighour of all~$S\in\mathcal
S$, i.e.\ such that $T\intersect\Union\mathcal S=\emptyset$.

Assume that there is an $x\in\Intersection_{S\in\mathcal S}
U_S$.  As we have seen in \prettyref{def:C-to-Hom} there
is a $T\in V(SG_{n,k})$ such that $\inprod x{(-1)^j v_j}<0$ for
all~$j\in T$ and therefore $T\intersect\Union\mathcal S=\emptyset$.

Now assume that there is a $T\in V(SG_{n,k})$ such that
$T\intersect\Union\mathcal S=\emptyset$.  Then we can set
\[s_j=\begin{cases}(-1)^j,&j\notin T,\\(-1)^{j+1},&j\in T.\end{cases}\]
The sign vector~$s$ has a sign change exactly for every $j$ with
$j,j+1\notin T$ and therefore at most $k$ of them, which means that
$s\in\mathcal L(C^{m,k+1})$.  Hence there is an $x\in\Sphere^k$ such
that $\sign\inprod x{v_j}=s_j$ for all~$j$.  In particular
$\inprod x{(-1)^jv_j}>0$ for all $j\in S\in\mathcal S$ and therefore
$x\in\Intersection_{A\in\mathcal S} U_S$.
\end{proof}
We have already noted in \prettyref{rem:barany} that $\Union_{S\in
  V(SG_{n,k})}U_S=\Sphere^k$.  Also, every intersection of sets $U_S$
is convex and hence contractible if non-empty.  By the nerve lemma, it
follows that $\real{\nc(SG_{n,k})}\homot\Sphere^k$.  To prove
\prettyref{thm:homotopy}, it remains to show that we not only have a
homotopy equivalence, but a $D_{2m}$-homotopy equivalence, and that
the map that we have constructed earlier is such an equivalence.  For
this, we replace the nerve lemma by an equivariant fibre lemma.  We
use the following from~\cite{equivariant-fibre}.

\begin{lem}[{\cite[Thm.~1]{equivariant-fibre}}]\label{lem:equivariant-fibre}
Let $\Gamma$ be a group and
$f\colon Q\to P$ an equivariant map between right
$\Gamma$-posets.  Assume that
for each $p\in P$ the space $\real{f^{-1}[[p,\top)]}$ is 
$\Gamma_p$-contractible,
where $[p,\top)=\set{p'\in P\colon p'\ge p}$ and 
$\Gamma_p=\set{\gamma\in\Gamma\colon p\gamma=p}$ is the stabilizer of~$p$.
Then $\real f\colon \real Q\to\real P$ is a $\Gamma$-homotopy
equivalence.
\qed
\end{lem}

We will apply it in the following form.

\begin{cor}\label{cor:equivariant-fibre}
Let $X$ be a completely regular finite cell complex $X$ with a right
action of a group $\Gamma$.  Let $P$ be a poset with a
right $\Gamma$-action.  Let $f\colon FX\to P$ be a $\Gamma$-equivariant
poset map from the
face poset of~$X$ to~$P$.  Denote by $\tilde f\colon X\to P$ the function 
which maps every point in the open cell $c$ to $f(c)$.
If for every $p\in P$ and $\Gamma_p$ the stabilizer of~$p$ the 
open set $\tilde f^{-1}[[p,\top)]$ is $\Gamma_p$-contractible,
then the map $\real f\colon \real X\to\real P$ is a $\Gamma$-homotopy
equivalence.
\end{cor}

\begin{proof}
Since $X$ is completely regular, it is homeomorphic to its barycentric
subdivision~$\real{FX}$.  This is how $f$ defines a map $\real f\colon
\real X\to\real P$.  To reduce the corollary to the lemma, it suffices
to show that 
$\tilde f^{-1}[[p,\top)]\homot_{\Gamma_p}\real{f^{-1}[[p,\top)]}$ 
for each~$p$.  We fix $p$ and set $R=f^{-1}[[p,\top)]$.  Now 
$\tilde f^{-1}[[p,\top)]$, seen as subset of the barycentric subdivision
of $X$, is the union of all open simplices with at least one vertex in~$R$.
Let $q_0<\dots<q_s$ be the vertices of such a simplex and 
$j=\min\set{i\colon q_i\in R}$.  Then
\begin{align*}
\left(\real{\set{q_0,\dots,q_s}}\intersect\tilde f^{-1}[[p,\top)]\right)
\times I &\to\real{\set{q_0,\dots,q_s}}\\
\left(\sum_i\lambda_i q_i,t\right)
&\mapsto
(1-t)\sum_{i}\lambda_i q_i
+t\left(\sum_{i\ge j}\lambda_i\right)^{-1}\sum_{i\ge j}\lambda_i q_i
\end{align*}
is a strong deformation retraction to~$\real{q_j,\dots,g_s}$.
These maps fit together to define a $\Gamma_p$-equivariant
strong deformation retraction
from $\tilde f^{-1}[[p,\top)]$ to $\real{f^{-1}[[p,\top)]}$.
\end{proof}

\begin{proof}[Proof of \prettyref{thm:homotopy}]
We denote by $\nc(G)$ the neighbourhood complex of a graph~$G$ 
or the face poset of this complex.  It is well-known that
the poset map 
\begin{align*}
\Hom(K_2, G)&\to\nc(G)\\
g&\mapsto g(0)
\end{align*}
induces a homotopy equivalence 
$\real{\Hom(K_2, G)}\to\real{\nc(G)}$, see \cite{csorba-thesis}.  
All steps in the proof given in~\cite[Ex.~3.3]{c5} are natural with respect
to automorphisms of~$G$, hence it is an
$\Aut(G)$-homotopy equivalence.
It will therefore suffice to show that the map
\begin{align*}
h\colon\mathcal L(C^{m,k+1})&\to\nc(SG_{n,k})\\
s&\mapsto\set{S\in V(SG_{n,k})\colon s_j=(-1)^j\text{ for all $j\in S$}}
\end{align*}
with $m=2n+k$ induces a $D_{2m}$-homotopy equivalence $\real
h\colon\Spherenk\to\real{\nc(SG_{n,k})}$.  Let $\mathcal S\in \mathcal
N(SG_{n,k})$.  Then, with notation as in 
\prettyref{cor:equivariant-fibre},
\begin{align*}
\tilde h^{-1}[[\mathcal S,\top)]
&=
\set{x\in\Spherenk\colon (-1)^j\inprod x{v_i}> 0\text{ for all
    $j\in\Union\mathcal S$}}
\\
&=\Intersection_{S\in\mathcal S}U_S,
\end{align*}
where the vectors $v_i$ are those used in the proof of 
\prettyref{thm:op-on-box}
and the sets $U_S$ are as in \prettyref{prop:nerve}.  We have seen
there that $\Intersection_{S\in\mathcal S}U_S\ne\emptyset$.  Let
$\Gamma_{\mathcal S}=\set{\gamma\in D_{2m}\colon \mathcal S\cdot
  \gamma=\mathcal S}$.  The group $\Gamma_{\mathcal S}$ acts by linear
maps on the convex set $\Intersection_{S\in\mathcal S}U_S$.
Consequently, the barycentre of any orbit of~$\Gamma_{\mathcal S}$ in
$\Intersection_{S\in\mathcal S}U_S$ is a fixed point of
$\Gamma_{\mathcal S}$, and $\Intersection_{S\in\mathcal S}U_S$
contracts equivariantly to it.  It follows from
\prettyref{cor:equivariant-fibre} that $\real h$ is a
$D_{2m}$-homotopy equivalence.
\end{proof}

\section{Groups and bundles}
The approach to show that for certain $n$, $k$ the graph $SG_{n,k}$ is
a homotopy test graph is similar to that used for odd cycles ($k=1$)
in \cite{hom-loop}.  If $\Hom(SG_{n,k}, G)$ is $(r-1)$-connected, then
there exists an equivariant map $\E_r D_{2m}\to_{D_{2m}}\Hom(SG_{n,k},
G)$.  In \prettyref{thm:op-on-box} we have constructed a map
$\Spherenk\to_{D_{2m}}\to_{D_{2m}}\Hom(K_2, SG_{n,k})$.  It is known
that $\Hom(K_2, K_{k+r+1})\homot_{C_2}\Sphere^{k+r-1}$.  So if
$\chi(G)<k+r+2$, then we obtain a map
\begin{multline*}
\Spherenk\times_{D_{2m}}\E_r D_{2m}
\to_{C_2}
\Hom(K_2, SG_{n,k})\times_{D_{2m}}\Hom(SG_{n,k}, G)
\\\to_{C_2}
\Hom(K_2, G)
\to_{C_2}
\Hom(K_2, K_{k+r+1})\homot_{C_2}\Sphere^{k+r-1}.
\end{multline*}
We will identify obstructions to the existence of such maps, compare
\prettyref{prop:test-graph}.
Later
we will see that under some circumstances the vanishing of one of
these obstructions implies that $SG_{n,k}$ is not a test graph.  In
this section we collect the topological tools that we will use.

\subsection{Cohomology of $D_{2m}$}\label{sec:cohomology}
We fix some notation and collect some facts regarding the cohomology
of the groups that we will use.

We recall
\begin{align*}
D_{2m}&=\left\langle \sigma,\rho \,|\, \rho^2=\sigma^m=(\sigma\rho)^2=1
\right\rangle,
\\
C_m&=\left\langle \sigma\,|\,\sigma^m=1\right\rangle.
\end{align*}

The calculation of the cohomology groups with coefficients in~$\Z_2$ of these
groups can be reduced to the case where $m$~is a power of~$2$.
\begin{lem}\label{lem:tm}
Let $m,t\ge1$ and $t$ odd.  Then the group monomorphisms $C_m\to
C_{tm}$ and $D_{2m}\to D_{2tm}$ given by $\sigma\mapsto\sigma^t$, and
$\rho\mapsto\rho$ for the second monomorphism, induce isomorphisms
$H^\ast(C_{m};\Z_2)\xto\isom H^\ast(C_{tm};\Z_2)$ and
$H^\ast(D_{2m};\Z_2)\xto\isom H^\ast(D_{2tm};\Z_2)$.
\end{lem}

\begin{proof}
First note that $H^i(C_m;\Z_2)$ is isomorphic to $0$ or~$\Z_2$ for
all~$i$ and that $H^i(C_t;\Z_2)=0$ for $i>0$.  Applying the
Lyndon-Hochschild-Serre spectral sequence to
\[\xymatrix{
C_m\ar@{>->}[r]&
C_{tm}\ar@{->>}[r]&
C_t
}\]
now yields that $H^\ast(C_{mt};\Z_2)\xto\isom H^\ast(C_m;\Z_2)$.

Now $C_m$ can be regarded as an index-$2$ subgroup of $D_{2m}$ via
the inclusion map $\sigma\mapsto\sigma$.
Applying the Lyndon-Hochschild-Serre spectral sequence to the rows of
\[\xymatrix{
C_m\ar@{>->}[r]\ar[d]&
D_{2m}\ar@{->>}[r]\ar[d]&
C_2\ar@{=}[d]
\\
C_{tm}\ar@{>->}[r]&
D_{2tm}\ar@{->>}[r]&
C_2
}\]\sloppy
shows that the induced map between the $E^2$-terms
$H^i(C_2; H^j(C_{tm};\Z_2))\to H^i(C_2; H^j(C_m;\Z_2))$ is an isomorphism
and therefore also the map $H^\ast(D_{2m};\Z_2)\to H^\ast(D_{2tm};\Z_2)$.
\end{proof}

We have
\begin{align*}
H^\ast(C_2;\Z_2)&\isom\Z_2[\alpha],&&\adeg\alpha=1.
\end{align*}

For an element $g$ of a group $G$ with $g^2=1$, let us denote the
homomorphism $\tau\mapsto g$ by $\phi_g\colon C_2\to G$.  

\subsubsection{The case $m\equiv 1\pmod 2$}
Since $D_2=\set{e,\rho}$ it follows from \prettyref{lem:tm} that
$\phi_\rho\colon C_2\to D_{2m}$
induces an isomorphism $H^\ast(D_{2m};\Z_2)\xto\isom H^\ast(C_2;\Z_2)$.

\subsubsection{The case $m\equiv 2\pmod 4$}
The group $D_4=\set{e,\rho,\sigma,\sigma\rho}$ is isomorphic to 
$C_2\times C_2$.  It follows that 
\begin{align*}
H^\ast(D_{2m};\Z_2)&\isom\Z_2[\alpha,\beta],&& \adeg\alpha=\adeg\beta=1
\end{align*}
with
\begin{align*}
  \phi_{\sigma^{m/2}}^\ast\colon H^\ast(D_{2m};\Z_2)&\to H^\ast(C_2;\Z_2)&\phi_{\rho}^\ast\colon H^\ast(D_{2m};\Z_2)&\to H^\ast(C_2;\Z_2)\\
\alpha&\mapsto\alpha&\alpha&\mapsto0\\
\beta&\mapsto0&\beta&\mapsto\alpha.
\end{align*}

\subsubsection{The case $m\equiv 0\pmod 4$}
We have
\begin{align*}
H^\ast(C_m;\Z_2)&\isom\Z_2[x,u]/(x^2),&& \adeg x=1, \adeg u=2,
\\
H^\ast(D_{2m};\Z_2)&\isom\Z_2[x,y,u]/(xy),&& \adeg x=\adeg y=1, \adeg u=2.
\end{align*}
The group homomorphisms
\begin{equation}\label{eq:p-and-j}
\begin{aligned}
p\colon C_{m}&\to C_2
\qquad
&j\colon C_m&\to D_{2m}\\
\sigma&\mapsto\tau&\sigma&\mapsto\sigma
\end{aligned}
\end{equation}
induce homomorphisms
\begin{align*}
p^\ast\colon H^\ast(C_2;\Z_2)&\to H^\ast(C_m;\Z_2)&
j^\ast\colon H^\ast(D_{2m};\Z_2)&\to H^\ast(C_m;\Z_2)\\
\alpha&\mapsto x&
x&\mapsto x\\
&&y&\mapsto x\\
&&u&\mapsto u.
\end{align*}
We also have
(all coefficient groups $\Z_2$)
\begin{align*}
\phi_{\rho}^\ast\colon H^\ast(D_{2m})&\to H^\ast(C_2)&
\phi_{\sigma\rho}^\ast\colon H^\ast(D_{2m})&\to H^\ast(C_2)&
\phi_{\sigma^{m/2}}^\ast\colon H^\ast(C_m)&\to H^\ast(C_2)\\
u&\mapsto0&u&\mapsto0&u&\mapsto{\textstyle\alpha^2}\\
x&\mapsto\alpha&x&\mapsto0&x&\mapsto0\\
y&\mapsto0&y&\mapsto\alpha
\end{align*}
These follow from \prettyref{lem:tm} and the calculations in
\cite[IV.2]{adem-milgram}.

\subsection{The index of a group action}
We will formulate some results using ideal valued index of Fadell and
Husseini, see \cite{guideII}.  We will only
consider for cohomology with coefficients in~$\Z_2$.  

We assume all spaces to be CW-spaces and $\Gamma$ to be a finite group.

\begin{defn}\label{def:Er}
By $\E\Gamma$ we denote a contractible free $\Gamma$-space.

By $\E_r\Gamma$, $r\ge0$, 
we denote any $(r-1)$-connected, $r$-dimensional 
free $\Gamma$-CW-complex, or, if we need a specific example,
the simplicial complex
$\Gamma\join\cdots\join\Gamma$ ($r+1$ factors), which satisfies 
these properties.

We write $\B\Gamma=\E\Gamma/\Gamma$, $\B_r\Gamma=\E_r\Gamma/\Gamma$.
\end{defn}

\begin{defprop}\label{def:Ind}
If $X$ is a space
with an action of a group~$\Gamma$, the index
$\Ind_\Gamma(X)\subset H^\ast(\Gamma;\Z_2)$ is defined as the kernel
of the map 
\[H^\ast(\Gamma;\Z_2)=H^\ast_\Gamma(\ast;\Z_2)
\to H^\ast_\Gamma(X;\Z_2)\]
induced by the map from $X$ to a one-point space.  Here
$H_\Gamma^\ast(X;\Z_2)\deq H^\ast((\E\Gamma\times X)/\Gamma; \Z_2)$,
where $\E\Gamma\times X$ carries the diagonal action of~$\Gamma$.
If $X$ is a free gamma space, then 
$H^\ast_\Gamma(X;\Z_2)\isom H^\ast(X/\Gamma; \Z_2)$ and $\Ind_\Gamma(X)$
equals the kernel of the map
\[H^\ast(\Gamma; \Z_2)
=H^\ast(\B\Gamma;\Z_2)\xto{{\bar f}^\ast} H^\ast(X/\Gamma;\Z_2),\]
where $\bar f\colon X/\Gamma\to\B\Gamma$ is
induced by an equivariant map
$f\colon X\to_\Gamma\E\Gamma$.
\end{defprop}

\begin{prop}\label{prop:Ind-mon}
Let $f\colon X\to_\Gamma Y$ be an equivariant map between $\Gamma$-spaces.
Then $\Ind_\Gamma(Y)\subset\Ind_\Gamma(X)$.
\qed
\end{prop}

\begin{prop}\label{prop:Ind-coind}
Let $r\ge0$ and $X$ be a $\Gamma$-space  
such that there is an equivariant map $\E_r\Gamma\to_\Gamma X$.
Then $\Ind_\Gamma(X)\intersect H^r(\Gamma)=0$.
\end{prop}

\begin{proof}
We can construct a free $\Gamma$-space $\E\Gamma$
such that $\E_r\Gamma\subset \E\Gamma$ and such that
$\B\Gamma=\E\Gamma/\Gamma$ is obtained from
$\B_r\Gamma=E_r\Gamma/\Gamma$ by attaching cells of dimension greater
than~$r$.  This shows that the map $H^r(\B\Gamma)\to H^r(\B_r\Gamma)$
is a monomorphism.  Now since $X$ is $(r-1)$-connected, there is an equivariant
map $E_r\Gamma\to_\Gamma X$.  This shows
\[\Ind_\Gamma(X)\intersect H^r(\Gamma)\subset
\Ind_\Gamma(E_r\Gamma)\intersect H^r(\Gamma)
=0\]
as claimed.
\end{proof}

\subsection{Stiefel-Whitney classes}
\begin{thm}\label{thm:sw}
Let $\xi=(E\to X)$ be a real
$(k+1)$-dimensional vector bundle and let $S(\xi)=(S(E)\to X)$ be the
associated $k$-sphere-bundle and $p\colon P(E)\to X$ its
projectivization.  Let $t\in H^1(P(E);\Z_2)$ be the first
Stiefel-Whitney class of (the real line bundle associated to) the
antipodal $C_2$-action on $S(E)$.  We regard $H^\ast(P(E);\Z_2)$ as a module 
over $H^\ast(X; Z_2)$ via the homomorphism~$p^\ast.$

$H^\ast(P(E);\Z_2)$ is a free module over $H^\ast(X; Z_2)$ with basis
$(1,t,t^2,\dots,t^k)$ and
\[\sum_{r=0}^{k+1}w_{k+1-r}(\xi)t^r=0.
\]
\end{thm}

\begin{proof}
This property characterizes the Stiefel-Whintney classes,
see \cite[Chap.~17]{huse}. 
\end{proof}

\begin{lem}\label{lem:wbar}
In the situation of \prettyref{thm:sw}, if we express $t^{k+\ell}$ for
$\ell\ge0$ as
\[t^{k+\ell}=\sum_{j=0}^k s_{j,\ell+k-j}t^j\]
with $s_{j,r}\in H^r(X;\Z_2)$, then $s_{k,r}=\overline w_r(\xi)$, where $\overline
w(\xi)=\sum_{r\ge 0}\overline w_r(\xi)$ is defined by $\overline w(\xi)w(\xi)=1$.
\end{lem}

\begin{proof}
Setting $s_j=\sum_{r\ge0}s_{j,r}$, the defining relation for the~$s_j$ is
\[\sum_{\ell\ge0}t^{k+\ell}=\sum_{j=0}^k s_jt^j.\]
Multiplying by $1-t$ yields
\begin{align*}
t^k&=\sum_{j=0}^ks_jt^j-\sum_{j=0}^ks_jt^{j+1}
%\\&
=\sum_{j=0}^ks_jt^j-\sum_{j=0}^{k-1}s_jt^{j+1}+s_k\sum_{j=0}^kw_{k+1-j}(\xi)t^j.
\end{align*}
Since $(t^0,\dots,t^k)$ is a basis, we can add the coefficients on both sides,
which yields
\[1=s_k+s_k\sum_{j=0}^kw_{k+1-j}=s_k w(\xi)\]
as claimed.
\end{proof}

\begin{cor}\label{cor:sw}
Let $\Gamma$ be a group and $X$ a free $\Gamma$-space.  Let $\R^{k+1}$
be equipped with an orthogonal right $\Gamma$-action.  Let $\xi$ be
the $(k+1)$ dimensional bundle $\R^{k+1}\times_\Gamma \E\Gamma\to
\B\Gamma$.  Consider $\Sphere^k\times_\Gamma X$ as a (free) $C_2$-space
via the antipodal action on $\Sphere^k$.  Then if
$\alpha^{r+k}\in\Ind_{C_2}(\Sphere^k\times_\Gamma X)$ for some~$r\ge
0$, it follows that $\wbar_r(\xi)\in\Ind_\Gamma(X)$.
\end{cor}

\begin{proof}
Let $\eta$ be the bundle $\R^{k+1}\times_\Gamma X\to X/\Gamma$.
Using the notation from \prettyref{thm:sw} applied to the bundle~$\eta$,
$\alpha^{r+k}\in\Ind_{C_2}(\Sphere^k\times_\Gamma X)$ is equivalent to
$t^{r+k}=0$.  Therefore if we express $t^{r+k}$ in the basis
$t^0,\dots,t^k$ with coefficients in $H^\ast(X/\Gamma;\Z_2)$, all of
these coefficients will be zero.  By \prettyref{lem:wbar} the
coefficient of~$t^k$ is~$\wbar_r(\eta)$.  Now if 
$f\colon X\to_\Gamma\E\Gamma$ is an equivariant map, then since the class
$\wbar_r$ is characteristic $f^\ast(\wbar_r(\xi))=
\wbar_r(\eta)=0$, i.e.\ $\wbar_r(\xi)\in\Ind_\Gamma(X)$.
\end{proof}

\begin{prop}\label{prop:wbar-obs}
Let $\xi$ be a vector bundle over an $r$-dimensional finite simplicial 
complex~$X$ and $\wbar_r(\xi)=0$.  If $r=1$ or $r\equiv0\pmod2$, then
there is an $(r-1)$-dimensional vector bundle $\eta$ over~$X$ such that
$\xi\dplus\eta$ is trivial.
\end{prop}

\begin{proof}
Let $\xi$ be a $k+1$-dimensional bundle, $k>0$.
If $w_1(\xi)=\wbar_1(\xi)=0$, then $\xi$ is orientable and therefore trivial,
if $X$ is $1$-dimensional, since the special linear groups are connected.

We can therefore assume that $r$ is even.  Since $X$ is a finite
simplicial complex, there is an $n_1$-dimensional bundle $\eta_1$ such
that $\xi\dplus\eta_1\isom\eps_{k+1+n_1}$ is trivial.  We can assume
that $n_1>r$. In this situation, $w_r(\eta_1)$ is the only obstruction
against the existence of $n_1-r+1$ linear independent sections of
$\eta_1$.\cite[§12]{milnor-stasheff}  
Since $w_r(\eta_1)=\wbar(xi)=0$, there is an
$(r-1)$-dimensional bunde $\eta$ such that
$\eta\dplus\eps_{n_1+r+1}\isom\eps_{n_1}$.  Hence $\xi\dplus\eta$ is
stably trivial.  Since $k+1+r-1>r=\dim X$, it follows that $\xi\dplus\eta$ is
trivial.\cite[Lemma~3.5]{kervaire-milnor}
\end{proof}

\section{Calculation of Stiefel-Whitney classes}\label{sec:calc-sw}
\begin{defn}\label{def:xink}
Let $n\ge1$, $k\ge0$, $m=2n+k$.  We denote the $(k+1)$-dimensional vector bundle
$W_{n.k}\times_{D_{2m}}\E D_{2m}\to \B D_{2m}$ by $\xi_{n,k}$.
\end{defn}

In this section we will compute the total Stiefel-Whitney classes
\[w(\xi_{n,k})\in H^\ast(\B D_{2m};\Z_2)=H^\ast(D_{2m};\Z_2).\]  
These calculations will be needed in the following sections to
determine which stable Kneser graphs are test graphs.  The criteria
which require these calculations for their application are
\prettyref{prop:test-graph} and \prettyref{prop:not-test-graph}.

We will
have to distinguish the three cases corresponding to the three cases
in \prettyref{sec:cohomology}.  We will use notation from that
section, in particular the presentations of the $\Z_2$-algebras 
$H^\ast(D_{2m};\Z_2)$.

\begin{prop}\label{prop:odd-w}
Let $k=2r+1$, $r\ge0$, $n>0$, and $m=2n+k$.
Then
\begin{equation*}
w(\xi_{n,k})=(1+\alpha)^{r+1}\in H^\ast(D_{2m};\Z_2)\isom 
H^\ast(C_2;\Z_2).
\end{equation*}
\end{prop}

\begin{proof}
First note that $m=2(n+r)+1$ is odd.
We use that 
$\phi_\rho^\ast(w(\xi_{n,k}))=w(\phi_\rho^\ast(\xi_{n,k}))$ and that 
$\phi_\rho^\ast\colon H^\ast(D_{2m};\Z_2)\to H^\ast(C_m;\Z_2)$ is an isomorphism.
The element~$\rho\in D_{2m}$ acts on $W_{n,k}$ by $\diag(1,-1,\dots,1,-1)$.
Therefore $\phi_\rho^\ast(\xi_{n,k})$ is the Whitney sum of $k+1$ line bundles, half of them trivial, the other isomorphic to the canonical line bundle over
$\B C_2\homot\RP^\infty$.
It follows that $w(\phi_\rho^\ast(\xi_{n,k}))=(1+\alpha)^{(k+1)/2}$.
\end{proof}

\begin{prop}\label{prop:even-2-w}
Let $k=2r$, $r\ge1$, $n>0$, $n\equiv r+1\pmod 2$, and $m=2n+k$.
Then
\begin{multline*}
w(\xi_{n,k})=(1+\alpha)(1+\beta)^{\lceil r/2\rceil}
((1+\alpha)(1+\alpha+\beta))^{\lfloor r/2\rfloor}
\\\in H^\ast(D_{2m};\Z_2)\isom H^\ast(C_2\times C_2;\Z_2).
\end{multline*}
\end{prop}

\begin{proof}
We note that $m=2(n+r)\equiv 2\pmod4$.  By \prettyref{lem:tm} it
suffices to consider the bundle over $\B(C_2\times C_2)$ obtained by
restricting the action of $D_{2m}$ to the subgroup
$\set{e,\sigma^{m/2},\rho,\rho\sigma^{m/2}}$.  The element
  $\sigma^{m/2}$ acts by $\diag(-1,1,1,-1,-1,\dots)$ and the element
$\rho$ by $\diag(1,1,-1,1,-1,\dots)$ with the first
  entry exceptional and the others repeating with a period of four.
We can therefore write this bundle as a Whitney sum of line bundles
$\xi_0\dplus\xi_1\dplus\xi_2\dplus\dots\dplus\xi_k$.
We read off \sloppy
$w_1(\phi_{\sigma^{m/2}}^\ast(\xi_0))=
w_1(\phi_{\sigma^{m/2}}^\ast(\xi_{4j+3}))=
w_1(\phi_{\sigma^{m/2}}^\ast(\xi_{4j+4}))=\alpha$,
$w_1(\phi_{\sigma^{m/2}}^\ast(\xi_{4j+1}))=
w_1(\phi_{\sigma^{m/2}}^\ast(\xi_{4j+2}))=0$, and
$w_1(\phi_\rho^\ast(\xi_{4j+2}))=
w_1(\phi_\rho^\ast(\xi_{4j+4}))=
\alpha$,
$w_1(\phi_\rho^\ast(\xi_0))=
w_1(\phi_\rho^\ast(\xi_{4j+1}))=
w_1(\phi_\rho^\ast(\xi_{4j+3}))=
0$.
Therefore $w(\xi_0)=1+\alpha$, $w(\xi_{4j+1})=1$,
$w(\xi_{4j+2})=1+\beta$, $w(\xi_{4j+3})=1+\alpha$,
$w(\xi_{4j+4})=1+\alpha+\beta$.  The result follows.
\end{proof}

\begin{prop}\label{prop:even-4-w}
Let $k=2r$, $r\ge 1$, $n>0$, $n\equiv r\pmod 2$, $m=2n+k$.
Then
\[w(\xi_{n,k})=(1+y)(1+x+y+u)^{\lceil r/2\rceil}(1+x+y)^{\lfloor r/2\rfloor}
\in H^\ast(D_{2m};\Z_2).\]
\end{prop}

\begin{proof}
We note that $m=2(n+r)\equiv0\pmod4$.

We have $\xi_{n,k} =\xi_0\dplus\xi_1\dplus\xi_2\dplus\dots\dplus\xi_r$
according to the block structure of the matrices in
\prettyref{eq:def-sigma-even} and \prettyref{eq:def-rho-even}, where
$\xi_0$ is a line bundle and the other bundles are $2$-dimensional.

We have $\phi_\rho(w_1(\xi_0))=0$,
$\phi_{\sigma\rho}(w_1(\xi_0))=\alpha$, and hence $w_1(\xi_0)=y$.

Let $j>0$.  For the bundle $\xi_j$, both $\rho$ and $\sigma\rho$ act
as a reflection, and hence
$\phi_\rho^\ast(w_1(\xi_j))=\phi_{\sigma\rho}(w_1(\xi_j))=1+\alpha$,
and $w_1(\xi_j)=x+y$.  The element $\sigma^{m/2}$ acts by
$(-1)^j$, and hence 
$\phi_{\sigma^{m/2}}^\ast(w_2(\xi_j))=j\alpha^2$ and
$w_2(\xi_j)=ju$.
The equation
\[
w(\xi_j)=\begin{cases}
1+y,&j=0\\
1+x+y+u,&\text{$j>0$ odd,}\\
1+x+y,&\text{$j>0$ even}
\end{cases}
\]
follows.
\end{proof}

\section{Test graphs}\label{sec:test-graphs}
Using \prettyref{thm:op-on-box} we obtain following criterion, which will 
enable us to prove that certain stable Kneser graphs are test graphs.
\begin{prop}\label{prop:test-graph}
Let $n,k\ge 1$, $m=2n+k$, $r\ge1$, and let $G$ be a graph.  If
$\alpha^{r+k}\in\Ind_{C_2}(\Hom(K_2, G))$ then
$\wbar_r(\xi_{n,k})\in\Ind_{D_{2m}}(\Hom(SG_{n,k}, G))$.
\end{prop}

\begin{proof}
We consider the map
\[\Hom(K_2, SG_{n,k})\times_{D_{2m}}\Hom(SG_{n,k}, G)
\to_{C_2}\Hom(K_2, G),
\]
which extends composition of graph homomorphisms.  Now by
\prettyref{thm:op-on-box} there is an equivariant map
$\Sphere(W_{n,k})\to_{C_2\times D_{2m}}\Hom(K_2, SG_{n,k})$.  Since
the action of $D_{2m}$ on $\Hom(SG_{n,k}, G)$ may not be free, we
apply the Borel construction to this space and consider
$\Hom(SG_{n,k}, G)\times\E D_{2m}$ with the diagonal action
of~$D_{2m}$.  Projection onto any of the two factors is equivariant.
Combining these maps we obtain an equivariant map
\[\Sphere(W_{n,k})\times_{D_{2m}}(\Hom(SG_{n,k}, G)\times\E D_{2m})
\to_{C_2}\Hom(K_2, G).
\]
Therefore if 
$\alpha^{r+k}\in\Ind_{C_2}(\Hom(K_2, G))$ then
\[\alpha^{r+k}\in
  \Ind_{C_2}(\Sphere(W_{n,k})\times_{D_{2m}}(\Hom(SG_{n,k}, G)\times\E D_{2m}))\]
and, by \prettyref{cor:sw},
\[\wbar_r(\xi_{n,k})\in
\Ind_{D_{2m}}(\Hom(SG_{n,k}, G)\times\E D_{2m})=
\Ind_{D_{2m}}(\Hom(SG_{n,k}, G))
\]
as claimed.
\end{proof}

\begin{defn}\label{def:gamma-test-graph}
Let $T$ be a graph with a right action of a finite group~$\Gamma$.
We call $T$ a \emph{$\Gamma$-test graph}, if for all loopless
graphs~$G$ and
integers $r\ge0$ such that there exists an equivariant map
$\E_r\Gamma\to_\Gamma\real{\Hom(T, G)}$ the inequality
$\chi(G)\ge\chi(T)+r$ holds.
\end{defn}

\begin{prop}\label{prop:gamma-homotopy}
Let $T$ be a graph with a right action of a finite group~$\Gamma$.
If $T$ is a $\Gamma$-test graph, then $T$ is a homotopy test graph.
\end{prop}

\begin{proof}
If $\real{\Hom(T, G)}$ is $(r-1)$-connected, then there is an
equivariant map $\E_r\Gamma\to_\Gamma\real{\Hom(T, G)}$, since $E_r$ is an
$r$-dimensional free $\Gamma$-space.
\end{proof}

\begin{cor}[to \prettyref{prop:test-graph}]\label{cor:test-graph}
Let $n,k\ge 1$, $m=2n+k$.  
If $\wbar_r(\xi_{n,k})\ne0$ for all $r\ge1$, then $SG_{n,k}$ is a
$D_{2m}$-test graph.
\end{cor}

\begin{proof}
Let $G$ be a graph such that there is a map
$\E_r D_{2m}\to_{D_{2m}}\Hom(SG_{n,k}, G)$ for some $r\ge0$.
Then by \prettyref{prop:Ind-coind}
\begin{equation*}
\Ind_{D_{2m}}(\Hom(SG_{n,k}, G))\intersect H^r(D_{2m};\Z_2)
=0
\end{equation*}
and therefore $\wbar_r(\xi_{n,k})\notin
\Ind_{D_{2m}}(\Hom(SG_{n,k}, G))$.  By \prettyref{prop:test-graph}
this implies $\alpha^{r+k}\notin\Ind_{C_2}(\Hom(K_2, G))$.
It follows that 
\[\chi(G)\ge r+k+2=r+\chi(SG_{n,k}),\]
since $\alpha^{r+k}\in\Ind_{C_2}(\Hom(K_2, K_{r+k+1}))
=\Ind_{C_2}\Sphere^{r+k-1}$.
\end{proof}

We can also obtain a more precise result.

\begin{cor}[to \prettyref{prop:test-graph}]\label{cor:test-graph-gamma}
Let $n,k\ge 1$, $m=2n+k$, $\Gamma$ a subgroup of $D_{2m}$ and
$i\colon\Gamma\to D_{2m}$ the inclusion map.
If $i^\ast(\wbar_r(\xi_{n,k}))\ne0$ for all $r\ge1$, then $SG_{n,k}$ is a
$\Gamma$-test graph.
\end{cor}

\begin{proof}
From the existence of a map $\E_r\Gamma\to_{\Gamma}\Hom(SG_{n,k}, G)$
we obtain $i^\ast(\wbar_r(\xi_{n,k}))\notin\Ind_\Gamma(\Hom(SG_{n,k},G))$.  
But $i^\ast[\Ind_{D_{2m}}(\Hom(SG_{n,k}, G))]\subset
\Ind_\Gamma(\Hom(SG_{n,k}, G))$, and
therefore 
$\wbar_r(\xi_{n,k})\notin \Ind_{D_{2m}}(\Hom(SG_{n,k},G))$.  
Now the proof proceeds as before.
\end{proof}

When we show that $SG_{n,k}$ is a $\Gamma$-test graph, we will, 
partly to continue the tradition of previous results of
this kind, want $\Gamma$ to act freely on $\Hom(SG_{n,k}, G)$.
Even though we could deduce the freeness of the action a posteriori
from the lower bound on (the finite number) $\chi(G)$, we will prove
it independently using the following criterion.

\begin{lem}\label{lem:g}
Let $T$ be a graph on which the group~$\Gamma$ acts.  Then the
following are equivalent.
\begin{enumerate}
\item\label{it:lem-g-i}
For all loopless graphs~$G$ the induced action of~$\Gamma$ on
$\real{\Hom(T, G)}$ is free.
\item\label{it:lem-g-ii}
For every $\gamma\in\Gamma\wo\set1$ there are $v\in V(T)$ and $k\ge0$ such that
$(v,v\cdot\gamma^k)\in E(T)$.
\end{enumerate}
\end{lem}

\begin{proof}
``$\Longleftarrow$'': If the action is not free then there is an
  $f\in\Hom(T, G)$ and a $\gamma\in\Gamma\wo\set1$ such that $\gamma
  f=f$.  This means that for every $v\in V(T)$ we have
  $f(v\gamma)=f(v)$ and therefore $f(v\gamma^k)=f(v)$ for all $v$
  and~$k$.  But $f(v\gamma^k)=f(v)$ implies $(v,v\gamma^k)\notin
  E(T)$, since $G$ is loopless.

``$\Longrightarrow$'': Assume there is a $\gamma\in\Gamma\wo\set1$
  such that $(v,v\gamma^k)\notin E(T)$ for all $v$ and~$k$.  This
  means that the orbits of the action of the subgroup
  $\Gamma'\deq\set{\gamma^k}$ of~$\Gamma$ are independent sets.  The
  quotient map $q\colon T\to T/\Gamma'$ is therefore a graph
  homomorphism to a loopless graph, and $\gamma\cdot g=g$, so the action
  on $\Hom(T, T/\Gamma')$ is not free.
\end{proof}

\begin{thm}\label{thm:n,2}
Let $k\in\set{1,2}$, $n\ge1$, and $m=2n+k$.  Let
$\Gamma=\set{1,\rho}\isom C_2$ be the subgroup of $D_{2m}$ generated
by~$\rho$.  Let $G$ be a loopless graph such that $\Hom(SG_{n,k},
G)\ne\emptyset$.  Then $\Gamma$ acts freely on $\Hom(SG_{n,k}, G)$ and
for all $r\ge0$ we have that
\[\alpha^{r+k}\in\Ind_{C_2}(\Hom(K_2, G))
\text{\ \,implies\ \,}
\alpha^r\in\Ind_{C_2}(\Hom(SG_{n,k}, G)).\]
In particular, $SG_{n,k}$ is a $\Gamma$-test graph.
\end{thm}

\begin{cor}\label{cor:n,2}
Let $k\in\set{1,2}$, $n\ge1$, $r\ge0$.  If
$\Hom(SG_{n,k}, G)$ is $(r-1)$-connected, then
$\chi(G)\ge r+k+2$.
\qed
\end{cor}

\begin{rem}
The graph $SG_{n,1}$ is a cyclic graph with $2n+1$ vertices.  For $k=1$, 
\prettyref{cor:n,2} is thus the result
of~\cite{babson-kozlov-ii}.  Our proof of \prettyref{thm:n,2}
here is a reformulation of the
proof in~\cite{hom-loop}, using only slightly different tools.

The case $k=2$ was also considered in \cite[Ex.~4.15]{hom-loop}, but in a
less systematic way, and the group action considered there corresponds
to choosing the subgroup $\set{1,\sigma^{n+1}}$ instead of
$\set{1,\rho}$.  The reader can check that in the following proof that
choice works only for even~$n$, which is the case that was proven in
\cite{hom-loop}.  Also, there Kneser graphs were considered, but that
is a less significant difference, as the proof can easily be adapted
to stable Kneser graphs.
\end{rem}

\begin{proof}[Proof of \prettyref{thm:n,2}]
To see that the action on $\Hom(SG_{n,k},G)$ is free, we have to that
show there is a vertex of~$SG_{n,k}$ that is mapped to a neighbour
by~$\rho$.  Indeed, $\set{1,2,\dots,n}$ and 
$\set{1,\dots,n}\cdot\rho=\set{n+k,\dots,2n+k-1}$ are disjoint.

In view of \prettyref{prop:test-graph}, all that remains to be shown
is that the restriction of $\wbar_r(\xi_{n,k})$ to
$H^\ast(\Gamma;\Z_2)$ equals $\alpha^r$ for all $r\ge0$.  This is
equivalent to the restriction of $w(\xi_{n,k})$ equalling $1+\alpha$.
This follows from the calculations done in \prettyref{sec:calc-sw}.
Alternatively we notice that $\rho$ acts on $W_{n,k}$ by 
$\diag(1,-1)$ for $k=1$ or
$\diag(1,1,-1)$ for $k=2$.  
\end{proof}

\begin{thm}\label{thm:2s,4}
Let $s\ge1$, and $m=4(s+1)$.  Set $m=2^at$ with $t$~odd and 
let $\Gamma\isom C_{2^a}$ be the subgroup of $D_{2m}$ generated
by~$\sigma^t$.  Let $G$ be a loopless graph such that $\Hom(SG_{2s,4},
G)\ne\emptyset$.  Then $\Gamma$ acts freely on $\Hom(SG_{2s,4}, G)$ and
for any  $r\ge0$ and $\beta$ the non-zero element of 
$H^{r}(C_{2^a};\Z_2)$ we have that
\[\alpha^{r+4}\in\Ind_{C_2}(\Hom(K_2, G))
\text{\ \,implies\ \,}
\beta\in\Ind_\Gamma(\Hom(SG_{2s,4}, G)).\]
In particular, $SG_{2s, 4}$ is a $\Gamma$-test graph.
\end{thm}
\begin{cor}\label{cor:2s,4}
Let $s\ge1$, $r\ge0$, and $G$ a graph.
If $\Hom(SG_{2s,4}, G)$ is $(r-1)$-connected, then $\chi(G)\ge r+6$.
\qed
\end{cor}

\begin{proof}[Proof of \prettyref{thm:2s,4}]
\def\bar{\boldsymbol}
We use \prettyref{lem:g} to show that the action of~$\Gamma$ on
$\Hom(SG_{2s,4},G)$ is free.  Let $1\ne\gamma\in\Gamma$.  We can write
$\gamma=\sigma^{t2^bu}$ with $u$~odd and $b<a$.  Then
$\gamma^{2^{b-a-1}}=\sigma^{2^{a-1}tu}=\sigma^{\frac m2}=\sigma^{2s+2}$.  Now setting
$S=\set{2j\colon 0\le j<s}\unite\set{2j+2s+3\colon 0\le j<s}$
we have $(S,S\cdot \gamma^{2^{b-a-1}})\in E(SG_{2s,4})$.

We have calculated $w(\xi_{2s,4})$ in \prettyref{prop:even-4-w}.
If $j\colon \Gamma\to D_{2m}$ is the inclusion map of
\prettyref{eq:p-and-j}, then
\begin{align*}
j^{\ast}(w(\xi_{2s,4}))&=(1+x)(1+2x+u)(1+2x)=(1+x)(1+u)\in H^\ast(\Gamma;\Z_2)
\intertext{and}
j^{\ast}(\wbar(\xi_{2s,4}))&=(1+x)^{-1}(1+u)^{-1}
=(1+x)\sum_{i\ge0}u^i,
\end{align*}
i.e.\ $j^\ast(\wbar_{2\ell})(\xi_{2s,4}))=u^\ell$,
$j^\ast(\wbar_{2\ell+1}(\xi_{2s,4}))=xu^\ell$.
The proposition now follows from \prettyref{prop:test-graph}.
\end{proof}

\section{Borsuk graphs}
We introduce the Borsuk graphs of spheres and show their close
relationship to stable Kneser graphs.

\begin{defn}\label{def:borsuk}
Let $(X, d)$ be a metric space with an isometric 
$C_2$-action and $\eps>0$.  We
define the \emph{$\eps$-Borsuk graph of $X$}, $B_\eps(X)$, as follows. 
The vertex set of $B_\eps(X)$ is the
set of all points of~$X$ and $x\nbo y\iff d(x,\tau y)<\eps$.
\end{defn}

\begin{rem}
In \cite{csorba-box} and \cite{anton-cas} a discrete analogue of the
$\eps$-Borsuk graph has been investigated.  Let $L$ be a $C_2$-cell
complex together with a homeomorphism $\real L\homeo_{C_2} X$.  We
denote by $L^1$ the reflexive graph whose vertices are the $0$-cells
of~$L$ with two vertices being adjacent if there is a cell in~$L$ of
which both of them are faces.  We then construct the graph
$K_2\times_{C_2}L^1$, which is loopless if the action on $L$ is free.
We can identify the vertex $[(0,x)]$ of $K_2\times_{C_2}L^1$ with the
$0$-cell $x$ of~$L$.  Its neighbours are all the vertices in the
closed star of $\tau\cdot x$.  Therefore, if the diameter of each cell
of~$L$ is less then~$\eps$, we can regard $K_2\times_{C_2}L^1$ as a
subgraph of~$B_\eps(X)$.

On the other hand, if $\eps$ is such that each $\eps$-ball is
contained in the open star of a vertex of~$L$, and such an $\eps>0$
always exists if $X$ is compact, then we can assign to each point 
$x$ of~$X$ a vertex $v$ of~$L$ such that the $\eps$-ball around~$x$
is contained in the open star of~$v$, and this defines a graph homomorphism
$B_\eps(X)\to K_2\times_{C_2}L^1$.
\end{rem}

\begin{rem}\label{rem:B-to-SG}
In \prettyref{thm:op-on-box} we have constructed a $C_2$-equvariant
map $\Sphere^k\to\Hom(K_2, SG_{n,k})$ 
using a cell complex, arising from a hypersphere arrangment on $\Sphere^k$,
with face poset $\mathcal L(C^{m,k+1})$.
The map is defined via a poset map
\begin{equation}%\label{eq:LtoKSG}
\mathcal L(C^{m,k+1})\to_{C_2} \Hom(K_2,SG_{n,k})
\end{equation}
which maps atoms to atoms.  Restricting to these, we obtain an
equivariant graph homomorphsism
\begin{equation}\label{eq:LtoKSG}
\mathcal L(C^{m,k+1})^1 \to_{C_2} [K_2,SG_{n,k}].
\end{equation}
The first graph is defined as in the preceding remark, the second
graph is the exponential graph of functions from $K_2$ to~$SG_{n,k}$.
Matters like these are treated in more detail in~\cite{anton-cas}.
The graph homomorphism \prettyref{eq:LtoKSG} in turn defines a graph
homomorphism
\begin{equation}\label{eq:LKtoSG}
K_2\times_{C_2}\mathcal L(C^{m,k+1})^1 \to SG_{n,k}.
\end{equation}
This map is also easy to describe directly, using the notation of
\prettyref{def:C-to-Hom} it is given by
$[(l,s)]\mapsto S_l(s)$.
In the preceding remark we have seen that for small enough~$\eps$
there is a graph homomorphism 
\begin{equation}\label{eq:BKtoL}
B_\eps(\Sphere^k)\to K_2\times_{C_2}\mathcal L(C^{m,k+1})^1.
\end{equation}
Now the map~\prettyref{eq:LtoKSG} is also $D_{2m}$-equivariant, and
therefore so is~\prettyref{eq:LKtoSG}.  Also, the map
\prettyref{eq:BKtoL} can be chosen in such a way that it is
$D_{2m}$-equivariant.  Consequently, there is an equivariant graph
homomorphism
\begin{equation*}
B_\eps(\Sphere(W_{n,k}))\to_{D_{2m}} SG_{n,k},
\end{equation*}
if $\eps>0$ is small enough.  We will now work towards a graph homomorphism
in the other direction.
\end{rem}

\begin{lem}
Let $m=2n+k$, $S\in V(SG_{n,k})$ and $(v_i)_{0\le i<m}$ a system of
vectors in~$\R^{k+1}$ realizing $C^{m,k+1}$.  Then
\[\sum_{i\in S}(-1)^iv_i\ne0.\]
\end{lem}

\begin{proof}
Assume that $S\subset\set{0,\dots,m-1}$ and $\sum_{i\in
  S}(-1)^iv_i=0$.  Then $(s_i)_{0\le i<m}$ with $s_i=0$ for $i\notin
S$ and $s_i=(-1)^i$ for $i\in S$ is a vector of~$C^{m,k+1}$. Therefore
there are $j_0<j_1<\dots<j_{k+1}$ with $\set{j_s}\subset S$ and
$(-1)^{j_{s+1}}=(-1)^{j_s+1}$ for $0\le s\le k$.  If $S$ is a stable
set, i.e.\ one which does not contain consecutive elements, 
then this implies $m\ge 2\card S+k+1$ and hence $\card S<n$.
\end{proof}

This justifies the following definition.
\begin{defn}\label{def:v}
For  $S\in V(SG_{n,k})$ let 
\[v(S)\deq\frac{\sum_{i\in S}(-1)^iv_i}{\norm{\sum_{i\in S}(-1)^iv_i}},\]
where $v_i\in\R^{k+1}$ is as in the proof of~\prettyref{thm:op-on-box}.
\end{defn}

We will also need a continuous version of the preceding lemma.

\begin{lem}\label{lem:fourier}
Let $r\ge0$, $k=2r$. We set
\begin{align*}
w(t)&=(1,\exp(it), \exp(2it),\dots,\exp(rit))\in\R\times\C^r=\R^{k+1}.
\end{align*}
Then for all $0\le a_1\le a_2\le\dots\le a_k\le 2\pi$
we have
\begin{equation}\label{eq:fourier}
\sum_{s=0}^{k-1}(-1)^s\int_{a_s}^{a_{s+1}}w(t)\,dt
\ne0,
\end{equation}
where we set $a_0=a_k-2\pi$.
\end{lem}

\begin{proof}
We assume that the expression \prettyref{eq:fourier} is zero and
obtain for $1\le j\le r$
\begin{align*}
0&=\sum_{s=0}^{k-1}(-1)^s\int_{a_s}^{a_{s+1}}\exp(jit)\,dt
=\sum_{s=0}^{k-1}(-1)^{s+1}\frac ij(\exp(ija_{s+1})-\exp(ija_s))
\\&=\frac{2i}{j}\sum_{s=1}^k(-1)^s\exp(ji a_s).
\end{align*}
Since also $\sum_{s=1}^k(-1)^s=0$, we obtain
\begin{equation}\label{eq:proof-fourier}
\sum_{s=1}^k(-1)^s w(a_s)=0.
\end{equation}
Now the $a_s$ do not have to be distinct, but since considering the
first coordinate of the expression in \prettyref{eq:fourier} yields
\[0=\sum_{s=0}^{k-1}(-1)^s(a_{s+1}-a_s)=2\pi+2\sum_{s=1}^r(a_{2s-1}-a_{2s}),\] 
at least one number has to appear with an odd multiplicity in the
system $(a_s)_s$.  It follows that \prettyref{eq:proof-fourier} is a
nontrivial linear combination of vectors on the trigonometric moment
curve $w$, contradicting the linear independence of any $k+1$ such
vectors.
\end{proof}

\begin{prop}\label{prop:eps}
Let $k\ge0$ and $\eps>0$.  Then there is an $N\in\N$ such that
for all $n\ge N$ the function $v\colon V(SG_{n,k})\to\Sphere^k$ 
is a $D_{2m}$-equivariant graph homomorphism
\[SG_{n,k}\to_{D_{2m}} B_\eps(\Spherenk).\]
\end{prop}

\begin{proof}
Equivariance follows from \prettyref{eq:v-props}.  For the homomorphism
property, we have to show that there is an~$N$ such that
$\norm{v(S)+v(T)}<\eps$ for all $(S,T)\in E(SG_{n,k})$ with $n\ge N$.

We first show that
\begin{equation}\label{eq:andbeyond}
\lim_{n\to\infty}\min_{S\in V(SG_{n,k})}\bignorm{\sum_{r\in S}(-1)^rv_r}=\infty.
\end{equation}
For this it suffices to consider the case of an even~$k$, since for
$S\in V(SG_{n,2r+1})$ we can define $S'\in V(SG_{2n,s(2r+1)})$ by
$S'=\Union_{j\in S}\set{j,j+2(n+r)+1}$ and have $\bignorm{\sum_{j\in
    S'}(-1)^jv_j}=2\bignorm{\sum_{j\in S}(-1)^jv_j}$ (with different
$v_j$ on both sides of the equation) .  We therefore assume a fixed even~$k$.  
By \prettyref{lem:fourier} and compactness, there is a
$C>0$ such that
\begin{equation}\label{eq:gtC}
\bignorm{\sum_{s=0}^{k-1}(-1)^s\int_{a_s}^{a_{s+1}}w(t)\,dt}>C
\end{equation}
for all systems $(a_s)$ as considered there.
There is also a sequence $(\eps_m)$ converging to zero such that
\begin{equation}\label{eq:epsm}
\bignorm{\int_b^{b+4\pi/m}w(t)\,dt-\frac{4\pi}m w(b)}\le \frac{4\pi}m \eps_m
\end{equation}
for all $b$ and $m$.  We now consider an $S\in V(SG_{n,k})$ for
some~$n$ and set $m=2n+k$.  Since the norm in
\prettyref{eq:andbeyond} is invariant under the action of $D_{2m}$ on~$S$,
we may assume $0\in S$.
To~$S$ we associate a system $(a_s)$ as in 
\prettyref{lem:fourier} by requiring that
\[\set{a_s\colon 1\le s\le k}=\set{\frac{2\pi r}m\colon 
\text{$0<r\le m$ and $r-1,r-2\notin S$}}.\]
Now $[0,2\pi)$ is the disjoint union of the intervals
$[2\pi r/m,2\pi(r+2)/m)]$ for $r\in S$ and the intervals 
$[a_s-2\pi/m, a_s)$ for $1\le s\le k$. Also, for $r\in S$ with 
$a_s\le 2\pi r/m<a_{s+1}$ we have $(-1)^r=(-1)^s$.
It follows that
\begin{multline*}
\sum_{s=0}^{k-1}(-1)^s\int_{a_s}^{a_{s+1}}w(t)\,dt
=\\=
\sum_{r\in S}(-1)^r\int_{2\pi r/m}^{2\pi (r+2)/m}w(t)\,dt
+
\sum_{s=0}^{k-1}(-1)^s\int_{a_{s+1}-2\pi/m}^{a_{s+1}}w(t)\,dt
\end{multline*}
and hence (note that $w(2\pi r/m)=v_r$,
$\norm{w(t)}=\sqrt{k/2+1}$)
by \prettyref{eq:epsm}
\begin{multline*}
\bignorm{\frac{4\pi}m\sum_{r\in S}(-1)^rv_r
  -\sum_{s=0}^{k-1}(-1)^s\int_{a_s}^{a_{s+1}}w(t)\,dt}
\\\le
k\sqrt{k/2+1}\frac{2\pi}m+n\frac{4\pi}m\eps_m
\\=2\pi\left(\frac {k\sqrt{k/2+1}}{2n+k}+\frac{2n}{2n+k}\eps_{2n+k}\right)\xto[n\to\infty]{}0
,\end{multline*}
which implies
\begin{equation*}
\liminf_{n\to\infty} 
\min_{S\in V(SG_{n,k})}
\frac{4\pi}{2n+k}\bignorm{\sum_{r\in S}(-1)^rv_r}\ge C
\end{equation*}
and therefore \prettyref{eq:andbeyond}.

Now assume an arbitrary fixed~$k$ and let $N$ 
be such that for all $n\ge N$ we have 
\[\min_{S\in V(SG_{n,k})}\bignorm{\sum_{r\in S}(-1)^rv_r}
  >\frac{8k\sqrt{k/2+1}}\eps\]
and also for the action of $D_{2m}$ on $\Spherenk$ for $m=2n+k\ge 2N+k$
(compare \prettyref{eq:def-sigma-even}, \prettyref{eq:def-sigma-odd})
\[\norm{x\cdot\sigma+x}<\frac\eps2\qquad\text{for all $x\in\Spherenk$.}\]
For $(S,T)\in E(SG_{n,k})$ we have $\card{S\cdot\sigma\intersect T}\ge n-k$
and therefore 
\[\bignorm{\sum_{r\in S\cdot\sigma}(-1)^rv_r-\sum_{r\in T}(-1)^rv_r}
  \le 2k\sqrt{k/2+1}.\]
If $n\ge N$ then
\begin{align*}
\bignorm{v(S)+v(T)}
&\le\bignorm{v(S)+v(S)\cdot\sigma}+\bignorm{v(S\cdot\sigma)-v(T)}
\\&<\frac\eps2+ 
\frac{2\bignorm{\sum_{r\in S\cdot\sigma}(-1)^rv_r-\sum_{r\in T}(-1)^rv_r}}
   {\bignorm{\sum_{r\in S\cdot\sigma}(-1)^rv_r}}
\\&<\frac\eps2 +\frac{4k\sqrt{k/2+1}}{8k\sqrt{k/2+1}/\eps}
=\frac\eps2 + \frac\eps2=\eps
,\end{align*}
since $v(S\cdot\sigma)=v(S)\cdot\sigma$ and 
$\bignorm{\frac x{\norm x}-\frac y{\norm y}}\le \frac{2\norm{x-y}}{\norm x}$.
\end{proof}

\section{Construction of graph homorphisms}\label{sec:construction}
In this section we will show that most stable Kneser graphs $SG_{n,k}$
are not test graphs.  
\subsection{Vertex critical graphs}
We will use the following criterion.  Remember
that for a simplicial complex~$X$ we denote by $X^1$ its looped
$1$-skeleton, see \prettyref{sec:graph-complexes} and \cite{anton-cas}.

\begin{thm}\label{thm:crit-test}
Let $T$ be a finite 
graph equipped with a right action of a finite group~$\Gamma$.
Let $r\ge0$, $s\ge1$.  Then each of the following statements implies the next.
\begin{enumerate}
\item\label{it:ct-ii}
For every $\Gamma$-invariant triangulation $X$ of~$E_r\Gamma$
the inequality
\[
\chi(T\times_\Gamma X^1)\ge s
\]
holds.
\item\label{it:ct-iii}
There is no $\Gamma$-equivariant map
\[
\E_r\Gamma\to_\Gamma\real{\Hom(T, K_{s-1})}.
\]
\item\label{it:ct-iv}
For all graphs $G$ such that there is a $\Gamma$-equivariant map
$\E_r\Gamma\to\real{\Hom(T, G)}$ the inequality
\[\chi(G)\ge s\]
holds.
\item\label{it:ct-i}
For all graphs $G$ such that $\real{\Hom(T, G)}$ is $(r-1)$-connected
the inequality
\[\chi(G)\ge s
\]
holds.
\end{enumerate}
If $T$ is vertex critical and $\Gamma=\Aut(T)$, then \prettyref{it:ct-i}
implies \prettyref{it:ct-ii} and all of the statements are equivalent.
\end{thm}

\begin{cor}
If $T$ is a vertex critical graph with trivial automorphism group,
then $T$~is not a homotopy test graph.
\end{cor}

\begin{proof}
If $T$ is a homotopy test graph, then \prettyref{it:ct-i} holds
for $s=\chi(T)+r$.  However, if $\Gamma$ is trivial, then
\prettyref{it:ct-iii} holds for no $s>\chi(T)$, independent of $r$.
\end{proof}

\begin{cor}\label{cor:crit-test}
Let $T$ be a finite, vertex critical graph.  Then $T$ is an
$\Aut(T)$-test graph (\prettyref{def:gamma-test-graph})
if and only if $T$~is a homotopy test graph.
\end{cor}

\begin{proof}
Set $\Gamma=\Aut(T)$, $s=\chi(T)+r$, and consider the equivalence of
\prettyref{it:ct-iv} and \prettyref{it:ct-i}.
\end{proof}

\begin{proof}[Proof of \prettyref{thm:crit-test}]
``$\prettyref{it:ct-ii}\Longrightarrow\prettyref{it:ct-iii}$''
If there is a $\Gamma$-equivariant map
$\E_r\Gamma\to_\Gamma\real{\Hom(T, K_{s-1})}$, then for some
$\Gamma$-invariant triangulation $X$ of $\E_r\Gamma$ there is an equivariant
graph map
$X^1\to_\Gamma \Hom(T, K_{s-1})^1\to_\Gamma[T,K_{s-1}]$
and therefore a graph homomorphism
$T\times_\Gamma X^1\to K_{s-1}$.
 
``$\prettyref{it:ct-iii}\Longrightarrow\prettyref{it:ct-iv}$'' Assume
there is a $\Gamma$-equivariant map $\E_r\Gamma\to\real{\Hom(T, G)}$.
Every colouring  $G\to K_{s-1}$ induces
an equivariant map $\Hom(T, G)\to_\Gamma \Hom(T, K_{s-1})$, 
and composing these maps contradicts the assumption
\prettyref{it:ct-iii}.  
Therefore $\chi(G)\ge s$.

``$\prettyref{it:ct-iv}\Longrightarrow\prettyref{it:ct-i}$'' 
If $\real{\Hom(T, G)}$ is $(r-1)$-connected, then there is an equivariant map
$\E_r\Gamma\to\real{\Hom(T, G)}$.

``$\prettyref{it:ct-i}\Longrightarrow\prettyref{it:ct-ii}$'' We assume
  that a triangulation $X$ of $\E_r\Gamma$ is given.  If we obtain $Y$
  from~$X$ by repeated barycentric subdivision, then $Y$ is also
  $\Gamma$-invariant, and there is an equivariant graph homomorphism
  $Y^1\to_\Gamma X^1$.  If the subdivion $Y$ is fine enough,
  then \[ \real{\Hom(T, T\times_\Gamma Y^1)} \homot\real{\Hom(T,
  T)}\times_\Gamma\real{\Hom(T, Y^1)}\] by \cite[Sec.~5.2]{anton-cas}.
  Since we assumed $T$ to be vertex critical, the only endomorphisms
  of~$T$ are the automorphisms.  It also follows that $\Hom(T, T)$ is
  $0$-dimensional, $\Hom(T, T)\isom\Aut(T)$.  We also assumed
  $\Gamma=\Aut(T)$.  
  Therefore \[ \real{\Hom(T,
  T)}\times_\Gamma\real{\Hom(T, Y^1)}
  \homeo \Aut(T)\times_{\Aut(T)}\real{\Hom(T, Y^1)}
  \homeo\real{\Hom(T, Y^1)}.\] 
  But
  by \cite[Thm~3.1]{DocUni}, again if $Y$ is a fine enough
  subdivision, \[\real{\Hom(T, Y^1)}\homot\real Y\homeo E_r\Gamma.\]
  Therefore $\real{\Hom(T, T\times_\Gamma Y^1)}$ is $(r-1)$-connected.
  Hence $\chi(T\times_\Gamma X^1)\ge\chi(T\times_\Gamma Y^1)\ge s$.
\end{proof}

\subsection{General constructions}

\begin{prop}\label{prop:borsuk-to}
Let $k\ge0$, $r>0$.  Then there is an $\eps>0$ such that the following holds.

Let $\Gamma$ be a finite group which acts from the right on $\R^{k+1}$
by orthogonal maps.  Let $X$ be a finite simplicial complex with a
free $\Gamma$-action and let $\xi$ be the vector bundle
$\R^{k+1}\times_\Gamma\real{X}\to\real{X}/\Gamma$.  If 
there exists an $(r-1)$-dimensional vector bundle~$\eta$ over
$\real X/\Gamma$ such that $\xi\dplus\eta$ is trivial, then there is
a $\Gamma$-invariant subdivision $Y$ of~$X$ such that
\[\chi(B_\eps(\Spherenk)\times_\Gamma Y^1) < k+2+r.\]
\end{prop}

\begin{proof}
We choose a covering $(A_i)_{i=0,\dots,k+r}$ of $\Sphere^{k+r-1}$ by
closed subsets such that no~$A_i$ contains a pair of antipodal
points.  Let $D\deq\min_i\dist(A_i,-A_i)>0$, $0<\eps<D$ and $\eps'\deq
D-\eps$.

Since $\Gamma$ acts by orthogonal maps, the bundle $\xi$ is a
Euclidean vector bundle.  The bundle $\eta$ can be made into a
Euclidean vector bundle, and the $r+l$ linear independent sections of
$\xi\dplus\eta$ which define the trivialization can be made orthogonal
using Gram-Schmidt.  Therefore there is a continuous map
$E(\xi)=\R^{k+1}\times_\Gamma\real{X}\to\R^{k+r}$ such that the
restriction to each fibre of~$\xi$ is a linear isometry.  Denoting the
space of linear isomotries from $\R^{k+1}$ to $\R^{k+r}$ by
$\Iso(\R^{k+1},\R^{k+r})$ and viewing it as a $\Gamma$-space via the
action on~$\R^{k+1}$, this is equivalent to the existence of an
equivariant continuous map
\[f\colon \real{X}\to_\Gamma\Iso(\R^{k+1},\R^{k+r}).\]
We let $Y$ be a subdivision of~$X$ such that for all pairs $y,y'$ of
neighbouring vertices of~$Y$ we have $\norm{f(y)-f(y')}<\eps'$, where
$\norm\bullet$ denotes the operator norm. %

We define
\begin{align*}
c\colon \Spherenk\times V(Y)&\to\set{0,\dots,k+r},\\
(v,y)&\mapsto \min\set{i\colon f(y)(v)}.
\end{align*}
Now if $v,v'\in\Spherenk$, $\norm{v+v'}<\eps$, and $y,y'\in V(Y)$,
$\set{y,y'}\in Y$, then 
\begin{multline*}
\dist(A_{c(v,y)},-A_{c(v',y')})\le
\norm{f(y)(v)-(-f(y')(v'))}
\\\le\norm{f(y)-f(y')}+\norm{v-(-v')}<\eps'+\eps=D
\end{multline*}
and hence $c(v,y)\ne
c(v',y')$.  This shows that the function $c$ is a graph homomorphism
$B_\eps(\Spherenk)\times Y^1\to K_{k+r+1}$.  Since for
$\gamma\in\Gamma$ we have $f(\gamma y)(v)=(\gamma f)(v)=f(v\gamma)$,
we have $c(v,\gamma y)=c(v\gamma, y)$, and $[(v,y)]\mapsto c(v,y)$ 
defines a
$(k+1+r)$-colouring of $B_\eps(\Spherenk)\times_\Gamma Y^1$.
\end{proof}

\begin{prop}\label{prop:not-test-graph}
Let $k,r\ge1$ and $r=1$ or $r\equiv 0\pmod 2$.
Then there is an $N\ge 2$ such that for all $n\ge N$ with 
$\wbar_r(\xi_{n,k})=0$ there is a graph~$G$ such that $\Hom(SG_{n,k}, G)$
is $(r-1)$-connected and $\chi(G)<k+2+r$.
\end{prop}

\begin{proof}
Given $k$ and $r$ we choose $\eps>0$ as in \prettyref{prop:borsuk-to}
and $N$ as in \prettyref{prop:eps}.  Now given $n\ge N$, $m=2n+k$, 
there is an equivariant graph homomorphism
$SG_{n,k}\to_{D_{2m}} B_{\eps}(W_{n,k})$ by \prettyref{prop:eps}.

Now if $\wbar_r(\xi_{n,k})=0$ then by 
\prettyref{prop:wbar-obs} and \prettyref{prop:borsuk-to}
there is a $D_{2m}$-invariant triangulation $Y$ of $E_r D_{2m}$
such that $\chi(B_{\eps}(W_{n,k})\times_{D_{2m}} Y^1)<k+2+r$ and therefore
\[\chi(SG_{n,k}\times_{D_{2m}} Y^1)<k+r+2.\]
Since stable Kneser graphs are vertex critical with respect to the
chromatic number by a theorem of Schrijver~\cite{schrijver} and
$\Aut(SG_{n,k})={D_{2m}}$ for $n>1$ by a theorem of
Braun~\cite{braun-sg}, we can invoke \prettyref{thm:crit-test} to
conclude the proof.
\end{proof}

\subsection{Particular cases}
We now determine cases in which the condition
of \prettyref{prop:not-test-graph} is satisfied.
The results have been summarized in the introduction.

\begin{prop}\label{prop:even-4-wbar}
Let $k=2r$, $r\ge 3$, $n>0$, $n\equiv r\pmod 2$.
Then for
\[(\ell,\ell')=\begin{cases}
(3\cdot 2^a+1, 2^{a+2}),&\text{$r=2s$, $s\equiv 2^a\pmod{2^{a+1}}$, $a>0$},\\
(3\cdot 2^a-1, 2^{a+2}),&\text{$r=2s+1$, $s\equiv 2^a\pmod{2^{a+1}}$}, \\
(3\cdot 2^a-2, 2^{a+2}),&\text{$r=2s+2$, $s\equiv 2^a\pmod{2^{a+1}}$, $a>0$}
\end{cases}
\]
and $\ell\le d<\ell'$ we have 
\[\wbar_d (\xi_{n,k})=0.\]
The same is true for $r=2s+4$, $s\equiv 2^a\pmod{2^{a+1}}$, $a>1$ and
$(\ell,\ell')=(3\cdot 2^a-5,2^{a+2})$.
\end{prop}

\begin{cor}\label{cor:even-4-wbar}
Let $k=2r$, $r\ge 3$, $k\ne 8$, $n>0$, $n\equiv r\pmod 2$.  Then there is an
even $d\ge0$ such that $\wbar_d(\xi_{n,k})=0$.
\qed
\end{cor}

\begin{proof}[Proof of \prettyref{prop:even-4-wbar}]
We use \prettyref{prop:even-4-w}
and start by calculating sufficiently many terms of 
$\left((1+x+y+u)^s(1+x+y)^s\right)^{-1}$.
We have 
\begin{multline*}
\left((1+x+y+u)(1+x+y)\right)^{-1}
\\=
1+x^2+y^2+u+xu+yu\ +\text{ terms of degree $\ge4$}
\\=
(1+x+y)(1+x+y+u)\ +\text{ terms of degree $\ge4$}.
\end{multline*}
Since this implies that 
\begin{multline*}
\left((1+x+y+u)(1+x+y)\right)^{-2^{a'}}
=
1
\ +\text{ terms of degree $\ge2^{a+2}$}
\\\quad\text{for $a'>a$,}
\end{multline*}
we have
\begin{multline*}
\left((1+x+y+u)(1+x+y)\right)^{-s}\\=
(1+x+y)^{2^n}(1+x+y+u)^{2^n}
\ +\text{ terms of degree $\ge2^{a+2}$}
\\\quad\text{for $s\equiv 2^a\pmod{2^{a+1}}$.}
\end{multline*}
We calculate 
\begin{align*}
&(1+y)^{-1}(1+x+y)^{2^a}(1+x+y+u)^{2^a}
\\
&\qquad=((1+y)^{2^a-1}+x^{2^a})(1+x+y+u)^{2^a},
\\
&(1+y)^{-1}(1+x+y+u)^{-1}(1+x+y)^{2^a}(1+x+y+u)^{2^a}
\\
&\qquad=((1+y)^{2^a-1}+x^{2^a})(1+x+y+u)^{2^a-1},
\\
&(1+y)^{-1}(1+x+y+u)^{-1}(1+x+y)^{-1}(1+x+y)^{2^a}(1+x+y+u)^{2^a}
\\
&\qquad=((1+y)^{2^a-2}+(1+x)^{2^a-1}-1)(1+x+y+u)^{2^a-1},
\\
&(1+y)^{-1}(1+x+y+u)^{-2}(1+x+y)^{-2}(1+x+y)^{2^a}(1+x+y+u)^{2^a}
\\
&\qquad=((1+y)^{2^a-3}+(1+x)^{2^a-2}-1)(1+x+y+u)^{2^a-2},
\end{align*}
where we assume $a>0$ for the next to last equation
and $a>1$ for the last equation.
These expressions have terms of maximal degree $3\cdot 2^a$, $3\cdot 2^a-2$,
$3\cdot 2^a-3$, $3\cdot 2^a-6$ respectively.  We have shown that
$\overline w_d(\xi_{n,k})=0$ for $\ell\le d<\ell'$.
\end{proof}

\begin{prop}\label{prop:even-2-wbar}
Let $k=2r$, $r\ge2$, $n>0$, $n\equiv r+1\pmod 2$.
Then for
\[(\ell,\ell')=\begin{cases}
(3\cdot 2^a, 2^{a+2}),&\text{$r=2s$, $s\equiv 2^a\pmod{2^{a+1}}$},\\
(3\cdot 2^a-1, 2^{a+2}),&\text{$r=2s+1$, $s\equiv 2^a\pmod{2^{a+1}}$} \\
\end{cases}
\]
and $\ell\le d<\ell'$ we have 
\[\wbar_d (\xi_{n,k})=0.\]
The same is true for $r=2s+2$, $s\equiv 2^a\pmod{2^{a+1}}$, $a>0$ and
$(\ell,\ell')=(3\cdot 2^a-3,2^{a+2})$.
\end{prop}

\begin{cor}\label{cor:even-2-wbar}
Let $k=2r$, $r\ge3$, $n>0$, $n\equiv r+1\pmod 2$.
Then there is an
even $d\ge0$ such that $\wbar_d(\xi_{n,k})=0$.
\qed
\end{cor}

\begin{proof}[Proof of \prettyref{prop:even-2-wbar}]
We use \prettyref{prop:even-2-w}.
We have 
\begin{multline*}
((1+\beta)(1+\alpha)(1+\alpha+\beta))^{-1}
\\=1+\alpha^2+\alpha\beta+\beta^2+\alpha^2\beta+\beta\alpha^2 \ +\text{ terms of degree $\ge4$}
\\=(1+\beta)(1+\alpha)(1+\alpha+\beta) \ +\text{ terms of degree $\ge4$}.
\end{multline*}
It follows that
\[((1+\beta)(1+\alpha)(1+\alpha+\beta))^{-2^{a'}}
= 1\ +\text{ terms of degree $\ge 2^{a+2}$}
\text{ for $a'>a$}
\]
and
\begin{multline*}
((1+\beta)(1+\alpha)(1+\alpha+\beta))^{-s}
\\= 
(1+\beta)^{2^a}(1+\alpha)^{2^a}(1+\alpha+\beta)^{2^a}
+\text{ terms of degree $\ge 2^{a+2}$}
\\\text{ for $s\equiv 2^a\pmod{2^{a+1}}$}
\end{multline*}
Now
\begin{align*}
&(1+\alpha)^{-1}(1+\beta)^{2^a}(1+\alpha)^{2^a}(1+\alpha+\beta)^{2^a}
\\
&\qquad=(1+\beta)^{2^a}(1+\alpha)^{2^a-1}(1+\alpha+\beta)^{2^a}
,\\
&(1+\alpha)^{-1}(1+\beta)^{-1}(1+\beta)^{2^a}(1+\alpha)^{2^a}(1+\alpha+\beta)^{2^a}
\\
&\qquad=(1+\beta)^{2^a-1}(1+\alpha)^{2^a-1}(1+\alpha+\beta)^{2^a}
,\\
&(1+\alpha)^{-2}(1+\beta)^{-1}(1+\alpha+\beta)^{-1}
(1+\beta)^{2^a}(1+\alpha)^{2^a}(1+\alpha+\beta)^{2^a}
\\
&\qquad=(1+\beta)^{2^a-1}(1+\alpha)^{2^a-2}(1+\alpha+\beta)^{2^a-1},
\end{align*}
where we assume $a>0$ in the last equation.
These expressions have terms of maximal degree 
$3\cdot 2^a-1$, $3\cdot 2^a-2$, $3\cdot 2^3-4$ respectively.
We have shown that
$\overline w_d(\xi_{n,k})=0$ for $\ell\le d<\ell'$.
\end{proof}

\begin{prop}\label{prop:odd-wbar}
Let $k=2r+1$, $r>0$, $n>0$.
Let $r\equiv 2^a\pmod{2^{a+1}}$.
For $2^a\le d<2^{a+1}$ we have 
\[\wbar_d (\xi_{n,k})=0.\]
\end{prop}

\begin{cor}\label{cor:odd-wbar}
Let $k=2r+1$, $r>0$, $n>0$.
Then there is a $d>0$ which is even or equals~$1$
such that $\wbar_d(\xi_{n,k})=0$.
\qed
\end{cor}

\begin{proof}[Proof of \prettyref{prop:odd-wbar}]
We use \prettyref{prop:odd-w}.
We have 
\begin{equation*}
(1+\alpha)^{-1}
=1+\alpha\ +\text{ terms of degree $\ge2$}.
\end{equation*}
It follows that
\[(1+\alpha)^{-2^{a'}}
= 1\ +\text{ terms of degree $\ge 2^{a+1}$}
\text{ for $a'>a$}
\]
and
\begin{multline*}
(1+\alpha)^{-r}
= 
(1+\alpha)^{2^a}
+\text{ terms of degree $\ge 2^{a+1}$}
\\\text{ for $r\equiv 2^a\pmod{2^{a+1}}$.}
\end{multline*}
Hence
\begin{equation*}
(1+\alpha)^{-(r+1)}
= 
(1+\alpha)^{2^a-1}
+\text{ terms of degree $\ge 2^{a+1}$}
\end{equation*}
We have shown that
$\overline w_d(\xi_{n,k})=0$ for $2^a\le d<2^{a+1}$.
\end{proof}

\section{Final Remarks}

Our proofs that some stable Kneser graphs $SG_{n,k}$ are not test
graphs only yield this result for $n\ge N(k)$, with no upper bound on
$N(k)$.  To obtain upper bounds from the current proofs would
necessitate to find a lower bound for the norm of the expression in
the inequality \prettyref{eq:fourier} in the Fourier analysis
lemma~\ref{lem:fourier}, and to substitute this for the compactness
argument in the proof of \prettyref{prop:eps}.  We have not tried
this.  For some cases a more direct combinatorial proof might be
possible.  The case in which $SG_{n,k}$ fails to be a test graph in
the lowest possible dimension is $k\equiv3\pmod4$.  We make this case
explicit.
\begin{prop}\label{prop:3-not-test}
Let $k\equiv3\pmod4$.  Then there is an $N(k)$ such that for all
$n\ge N(k)$ there is a graph $G$ such that
$\Hom(SG_{n,k},G)$ is non-empty and path-connected
and $\chi(G)=\chi(SG_{n,k})=k+2$.
\end{prop}

\begin{proof}
This follows from \prettyref{prop:not-test-graph} and 
\prettyref{prop:odd-wbar}.
\end{proof}
This rests on the following fact, which is also the source of the bound~$N(k)$.
\begin{prop}\label{prop:3-not-test-low-level}
Let $k\equiv3\pmod4$.  Then there is an $N(k)$ such that for all $n\ge
N(k)$ there are a colouring $c\colon SG_{n,k}\to K_{k+2}$ and for each
$\gamma\in D_{2m}$, $m=2n+k$, a path of graph homomorphisms from $c$
to $\gamma\cdot c$ in the graph $[SG_{n,k},K_{k+2}]$ (or,
equivalently, in the cell complex $\Hom(SG_{n,k}, K_{k+2})$).
\end{prop}

\begin{proof}
These paths of graph colourings constitute, for $X$ a subdivision of
the $1$-dimensional complex $E_1D_{2m}$, an equivariant graph
homomorphism $X^1\to_{D_{2m}}[SG_{n,k},K_{k+2}]$ and therefore the colouring
$SG_{2n,k}\times_{D_{2m}}X^1\to K_{k+2}$ of the proof of
\prettyref{prop:not-test-graph}.  Also note that the vanishing of the
obstruction $\wbar_1(\xi_{n,k})$ in this case just means that
all elements of $D_{2m}$ operate on $\Spherenk$ by elements of~$\SO(k+1)$.
\end{proof}

This purely graph theoretic statement might have a purely combinatorial
proof which yields a concrete bound $N(k)$.  We do not rule out 
the possibility that
the proposition holds for~$N(k)=2$.

%\bibliographystyle{casea1}
%\bibliography{topology,combi}
\newcommand{\etalchar}[1]{$^{#1}$}

\end{document}